\documentclass[letterpaper, 11pt,  reqno]{amsart}

\usepackage[margin=1.1in,marginparwidth=1.5cm, marginparsep=0.5cm]{geometry}

\usepackage{amsmath,amssymb,amscd,amsthm,amsxtra, esint}

\usepackage{mathrsfs} 

\usepackage[implicit=true]{hyperref}

\usepackage{color} 
\usepackage{ulem}
\usepackage[makeroom]{cancel}

\allowdisplaybreaks[2]

\definecolor{gr}{rgb}   {0.,   0.69,   0.23 }
\definecolor{bl}{rgb}   {0.,   0.5,   1. }
\definecolor{mg}{rgb}   {0.85,  0.,    0.85}
\definecolor{yl}{rgb}   {0.8,  0.7,   0.}
\definecolor{or}{rgb}  {0.7,0.2,0.2}

\newtheorem{theorem}{Theorem} [section]

\newtheorem{lemma}[theorem]{Lemma}

\newtheorem{remark}[theorem]{Remark}

\DeclareMathOperator*{\supp}{supp}

\DeclareMathOperator{\Id}{Id}

\newcommand{\I}{\hspace{0.5mm}\text{I}\hspace{0.5mm}}

\newcommand{\noi}{\noindent}
\newcommand{\Z}{\mathbb{Z}}
\newcommand{\R}{\mathbb{R}}

\newcommand{\T}{\mathbb{T}}

\let\Re=\undefined\DeclareMathOperator*{\Re}{Re}
\let\Im=\undefined\DeclareMathOperator*{\Im}{Im}

\newcommand{\E}{\mathbb{E}}

\renewcommand{\L}{\mathcal{L}}

\newcommand{\F}{\mathcal{F}}

\newcommand{\al}{\alpha}
\newcommand{\be}{\beta}
\newcommand{\dl}{\delta}

\newcommand{\nb}{\nabla}

\newcommand{\Dl}{\Delta}
\newcommand{\eps}{\varepsilon}
\newcommand{\kk}{\kappa}
\newcommand{\g}{\gamma}

\newcommand{\ld}{\lambda}
\newcommand{\Ld}{\Lambda}
\newcommand{\s}{\sigma}

\newcommand{\ft}{\widehat}

\newcommand{\wt}{\widetilde}
\newcommand{\cj}{\overline}
\newcommand{\dx}{\partial_x}
\newcommand{\dt}{\partial_t}

\newcommand{\too}{\longrightarrow}

\newcommand{\ta}{\theta}

\renewcommand{\l}{\ell}
\renewcommand{\o}{\omega}
\renewcommand{\O}{\Omega}

\newcommand{\les}{\lesssim}

\newcommand{\jb}[1]
{\langle #1 \rangle}

\newcommand{\ind}{\mathbf 1}

\newcommand{\deff}{\stackrel{\textup{def}}{=}}

\newcommand{\N}{\mathbb{N}}

\newcommand{\NN}{\mathcal{N}}

\newtheorem*{ackno}{Acknowledgements}

\numberwithin{equation}{section}
\numberwithin{theorem}{section}

\newcommand{\Q}{\mathbf{Q}}
\newcommand{\PP}{\mathbb{P}}

\newcommand{\C}{\mathbb{C}}

\newcommand{\z}{\zeta}

\DeclareMathOperator{\Law}{Law}

\newcommand{\rhoo}{\vec{\rho}}

\newcommand{\W}{\mathcal{W}}
\newcommand{\U}{\mathcal{U}}

\newcommand{\dr}{\theta}
\newcommand{\Dr}{\Theta}

\newcommand{\Ha}{\mathbb{H}_a}

\makeatletter
\@namedef{subjclassname@2020}{%
  \textup{2020} Mathematics Subject Classification}
\makeatother

\newcommand{\Gdl}{\mathcal{G}_{\dl} }

\begin{document}
\baselineskip = 14pt

\title[Log-correlated Gibbs measures]
{A remark on Gibbs measures with  log-correlated Gaussian fields}

\author[T.~Oh, K.~Seong, and L.~Tolomeo]
{Tadahiro Oh, Kihoon Seong, and Leonardo Tolomeo}

\address{
Tadahiro Oh\\ School of Mathematics\\
The University of Edinburgh\\
and The Maxwell Institute for the Mathematical Sciences\\
James Clerk Maxwell Building\\
The King's Buildings\\
Peter Guthrie Tait Road\\
Edinburgh\\ 
EH9 3FD\\
United Kingdom}

\email{hiro.oh@ed.ac.uk}

\address{
Kihoon Seong\\
Department of Mathematical Sciences\\
Korea Advanced Institute of Science
and Technology\\ 
291 Daehak-ro, Yuseong-gu, Daejeon 34141, Republic of Korea, 
Max Planck Institute 
for Mathematics In The Sciences\\ 
04103 Leipzig, Germany, 
and
Department of Mathematics, Cornell University, 310 Malott Hall, Ithaca, NY, 14853, USA
}

\email{kihoonseong@cornell.edu}

\address{
Leonardo Tolomeo\\ School of Mathematics\\
The University of Edinburgh\\
and The Maxwell Institute for the Mathematical Sciences\\
James Clerk Maxwell Building\\
The King's Buildings\\
Peter Guthrie Tait Road\\
Edinburgh\\ 
EH9 3FD\\
United Kingdom\\ and
Mathematical Institute, Hausdorff Center for Mathematics, Universit\"at Bonn, Bonn, Germany}

\email{l.tolomeo@ed.ac.uk}

\subjclass[2020]{60H30, 81T08, 35Q53, 35Q55, 35L71}

\keywords{Gibbs measure; log-correlated Gaussian field; non-normalizability}

\begin{abstract}

We study Gibbs measures with log-correlated base Gaussian fields
on the $d$-dimensional torus. 
In the defocusing case, 
the construction of  such Gibbs measures
follows from Nelson's argument.
In this paper, we consider the focusing case
with a quartic interaction. 
Using  the variational formulation, 
we  prove non-normalizability of the Gibbs measure.
When $d = 2$, our argument provides an alternative proof of the non-normalizability result
for the focusing $\Phi^4_2$-measure 
by Brydges and Slade (1996).
Furthermore, we provide a  precise rate of divergence, 
where the  constant is characterized by the optimal constant
for a certain Bernstein's inequality on $\R^d$.
We also go over the construction of the focusing Gibbs measure  
with a cubic interaction.
In the appendices, 
we present  (a)~non-normalizability of the Gibbs measure
for the two-dimensional Zakharov system
and  (b)~the construction of 
focusing quartic Gibbs measures with smoother base Gaussian measures, 
 showing  
a critical nature of 
 the log-correlated Gibbs measure
with  a focusing quartic interaction.

\end{abstract}


\maketitle


\tableofcontents

\newpage

\section{Introduction}

\subsection{Log-correlated Gibbs measures}
\label{SUBSEC:1.1}

In this paper, we study 
the Gibbs measure 
$\rho$ 
on 
the $d$-dimensional torus on $\T^d = (\R/2\pi\Z)^d$, formally written as\footnote{In this introduction, we keep our discussion at a formal level and  do not worry about renormalizations.
While we keep the following discussion only to the real-valued setting, 
our results also hold in the complex-valued setting, 
where  $k \ge 4$ is an even integer and 
 $u^k$ in \eqref{Gibbs1}
is replaced by 
 $|u|^k$.
See Footnote \ref{FT1}.

Hereafter, 
we  use $Z$, $Z_{N}$, etc.~to denote various normalization constants
whose values may change line by line.}
\begin{align}
d\rho(u) = Z^{-1} \exp \bigg(\frac{\ld}{k} \int_{\T^d} u^k dx\bigg) d\mu(u), 
\label{Gibbs1}
\end{align}

\noi
where $k \geq 3$ is an integer and 
the coupling constant
$\ld \in \R\setminus \{0\}$ denotes the strength of interaction, 
which is repulsive (i.e.~ defocusing) when $\ld < 0$ and $k$ is even, and is attractive 
(i.e.~focusing) when $\ld > 0$ or $k$ is odd.\footnote{In this paper, by ``focusing'', we mean ``non-defocusing''.
Namely, 
$\ld > 0$ or $k$ is odd in \eqref{Gibbs1}.} 
Here, 
 $\mu$ is the log-correlated Gaussian free field on~$\T^d$, 
 formally given by 
\begin{align}
d \mu 
= Z^{-1} e^{-\frac 12 \| u\|_{H^{d/ 2} }^2    } du
& =  Z^{-1} \prod_{n \in \Z^d} 
e^{-\frac 12 \jb{n}^{d} |\ft u(n)|^2}   
d\ft u(n) , 
\label{gauss0}
\end{align}

\noi
where 
$\jb{\,\cdot\,} = (1+|\,\cdot\,|^2)^\frac{1}{2}$
and $\ft u(n)$  denotes the Fourier coefficient  of $u$. 
When $d = 2$, 
$\mu$ corresponds to 
the massive Gaussian free field on $\T^2$. 
Recall that this Gaussian measure $\mu$ is 
nothing but the induced probability measure under the map:\footnote{By convention, 
we endow $\T^d$ with the normalized Lebesgue measure $dx_{\T^d}= (2\pi)^{-d} dx$
so that we do not need to worry about the factor $2\pi$ in various places.
 For simplicity of notation, we use $dx$ to denote the standard Lebesgue measure $\R^d$ and the normalized Lebesgue measure on $\T^d$ in the following.}
\begin{equation} 
\o\in \O \longmapsto u(\o) = \sum_{n \in \Z^d } \frac{ g_n(\o)}{\jb{n}^\frac d2} e_n, 
\label{IV2}
\end{equation}

\noi
where  $e_n=e^{i n\cdot x}$
and 
$\{ g_n \}_{n \in \Z^d}$
is a sequence of mutually independent standard complex-valued
Gaussian random variables on 
a probability space $(\O,\F,\PP)$ 
conditioned that 
 $g_{-n} = \cj{g_n}$.\footnote{In particular, $g_0$ is a standard real-valued
 Gaussian random variable.
 When $n \in \N$,  
 $\Re g_n$ and $\Im g_n$ are real-valued Gaussian random variables
 with mean 0 and variance $\frac 12$.}
 See Remark \ref{REM:corr}.
 It is well known that a typical function $u$
 in the support of $\mu$ is merely a distribution
 and thus a renormalization 
 on the potential energy $\frac \ld k \int_{\T^d} u^k dx$
 is required 
 for the construction of the Gibbs measure $\rho$.

Our main goal in this paper is to study 
the Gibbs measure $\rho$ in \eqref{Gibbs1} 
in the focusing case.
In particular, 
we prove 
the non-normalizability of the focusing Gibbs measure $\rho$
with the quartic interaction ($\ld > 0$ and $k = 4$). See Theorem \ref{THM:1}.
We also present a brief discussion on  the construction of the Gibbs measure  
with the cubic interaction. 
See Theorem \ref{THM:2}.

 Before proceeding further, let us first go over the defocusing case:
 $\ld < 0$ and $k \geq 4$ is an even integer.
 When $d = 2$, 
 the defocusing Gibbs measure $\rho$ in \eqref{Gibbs1}
 corresponds  
to the well-studied $\Phi^k_2$-measure
whose construction 
follows
from 
 the hypercontractivity of the Ornstein-Uhlenbeck semigroup
 (see Lemma \ref{LEM:hyp}) and Nelson's estimate~\cite{Nelson}. 
See \cite{Simon, GlimmJ87, DT, OTh}.
For a general dimension $d \geq 1$, 
the same argument allows us to construct
the defocusing Gibbs measure~$\rho$ in~\eqref{Gibbs1}
for any $\ld < 0$ and any even integer $k \ge 4$.
Let us briefly go over the procedure.

Given $N \in \N$, we define the  frequency projector\footnote{We may also proceed with 
regularization  via mollification.}  $\pi_N$  by 
\begin{align}
\pi_N f = 
\sum_{ |n| \leq N}  \ft f (n)  e_n.
\label{pi}
\end{align}

\noi
For   $u$   as in \eqref{IV2}, 
set $u_N = \pi_N u$.
Then, for each fixed $x \in \T^d$, 
$u_N(x)$ is
a mean-zero real-valued Gaussian random variable with variance
\begin{align}
\s_N = \E\big[u_N^2(x)\big] = \sum _{|n| \le N} \frac1{\jb{n}^d}
\sim  \log N \too \infty, 
\label{sigma1}
\end{align}

\noi
as $N \to \infty$. Note that $\s_N$ is independent of $x \in \T^d$
in the current translation invariant setting.
We then define the renormalized power (= Wick power)
$:\! u_N^k\!:$
by setting
\begin{align}
:\! u_N^k (x) \!: \, \stackrel{\text{def}}{=} H_k(u_N(x); \s_N), 
\label{Wick1}
\end{align}

\noi
where 
$H_k(x;\s)$ is the Hermite polynomial of degree $k$
with a variance parameter $\s$
defined through the following generating function:\footnote{\label{FT1}In the complex-valued setting
(with even $k$), we use the Laguerre polynomial
$c_k L_{\frac k2} (|u_N|^2; \s_N)$ to define the Wick renormalization.
See \cite{OTh}.}
\begin{equation}
F(t, x; \s) \stackrel{\text{def}}{=}  e^{tx - \frac{1}{2}\s t^2} = \sum_{k = 0}^\infty \frac{t^k}{k!} H_k(x;\s).
\label{Herm0}
 \end{equation}
	
\noi
For readers' convenience, we write out the first few Hermite polynomials:
\begin{align*}
\begin{split}
& H_0(x; \s) = 1, 
\quad 
H_1(x; \s) = x, 
\quad
H_2(x; \s) = x^2 - \s,   
\quad
 H_3(x; \s) = x^3 - 3\s x.
\end{split}
\end{align*}

\noi
See, for example, \cite{Kuo}, 
for further properties of the Hermite polynomials.
We then define the following renormalized truncated potential energy:
\begin{align}
\begin{split}
R_N(u)=\frac \ld k\int_{\T^d}  :\! u_N^k \!:  dx, 
\end{split}
\label{Wick0}
\end{align}

\noi
where the coupling constant
$\ld < 0 $ denotes the strength of repulsive interaction.
A standard computation allows
us to show that 
 $\{R_N \}_{N \in \N}$
forms a Cauchy sequence in $L^p(\mu)$
for any finite $p \geq 1$, 
thus converging to some 
 random variable  $R(u)$:
\begin{equation}
\lim_{N\rightarrow\infty} R_N(u)=R(u)
\label{conv1}
\end{equation}

\noi
in 
$L^p(\mu)$ and almost surely
See, for example, Proposition 1.1 in \cite{OTh}.\footnote{The claimed almost sure convergence
follows form the $L^p(\O)$-convergence in 
\cite[Proposition 1.1]{OTh}
together with the Borel-Cantelli lemma.}

Define  the renormalized truncated  Gibbs measure $\rho_{N}$ by 
\begin{align*}
d \rho_{N} (u) = Z_{N}^{-1} e^{R_N(u)} d \mu(u).
\end{align*}

\noi
Then, 
a standard application of Nelson's estimate\footnote{One may also prove the uniform 
 exponential integrability bound \eqref{exp1}
 via the variational approach as in \cite{BG}, 
 using the Bou\'e-Dupuis variational formula (Lemma \ref{LEM:var3}).
 When $k$ is large, however, the combinatorial complexity 
for the variational approach
 may be  cumbersome, 
 while there is no such combinatorial issue in 
 the approach of~\cite{DT, OTh}.
} yields
the following 
uniform exponential integrability of the density;
given  any finite $ p \ge 1$, 
there exists $C_{p, d} > 0$ such that 
\begin{equation}
\sup_{N\in \N} \Big\| e^{R_N(u)}\Big\|_{L^p( \mu)}
\leq C_{p,d}  < \infty.
\label{exp1}
\end{equation}

\noi
See, for example, Proposition 1.2 in \cite{OTh}.
Then, the  uniform bound \eqref{exp1}
together with 
softer convergence in measure (as a consequence of \eqref{conv1})
 implies the following $L^p$-convergence of the density:
\begin{equation*}
\lim_{N\rightarrow\infty}e^{ R_N(u)}=e^{R(u)}
\qquad \text{in } L^p( \mu).
\end{equation*}

\noi
See, for example, Remark 3.8 in \cite{Tz08}.
This allows us to construct the defocusing Gibbs measure:
\begin{align*}
d\rho(u)= Z^{-1} e^{R(u)}d\mu(u)
\end{align*}

\noi
as a limit of the truncated defocusing Gibbs measure $\rho_N$.

As mentioned above, our main goal is to study the Gibbs measure
$\rho$ with the log-correlated Gaussian field $\mu$
in the focusing case ($\ld > 0$).
Before doing so, we present a brief discussion 
on dynamical problems
associated with these Gibbs measures
in Subsection \ref{SUBSEC:PDE}.
We then present the non-normalizability 
of the focusing log-correlated Gibbs measure 
with the quartic interaction (Theorem \ref{THM:1})
and the construction of the focusing log-correlated Gibbs measure 
with the cubic interaction (Theorem \ref{THM:2}).

\begin{remark}
\label{REM:corr}
\rm 

Recall from 
\cite[(4,2)]{AS} that 
the Green's function $G_{\R^d}$ for $(1 - \Dl)^\frac{d}{2}$
on $\R^d$
satisfies 
\begin{align}
G_{\R^d}(x) = -c_d  \log|x| + o(1) 
\label{G1}
\end{align}

\noi
as $x \to 0$ for some $c_d > 0$.
Here, in view of the translation invariance, 
we view $G$ as a function of one variable
through $G(x) \equiv G(x, 0)$.
It is  a smooth function on $\R^d \setminus \{0\}$
and  decays exponentially as $|x| \to \infty$;
see \cite[Proposition 1.2.5]{Gra2}.

Now, 
let $G$ be the Green's function
for $(1-\Dl)^\frac{d}{2}$ on $\T^d$.
Then, 
we have 
\begin{align}
G \deff (1-\Delta)^{-\frac{d}{2}} \dl_0 = \sum_{n\in\Z^d}\frac 1{\jb{n}^{d}}e_n
= \lim_{N \to \infty} \sum_{|n|\leq N}\frac 1{\jb{n}^{d}}e_n.
\label{G1a}
\end{align}

\noi
Recall the Poisson summation formula (\cite[Theorem 3.2.8]{Gra1}):
\begin{align}
\sum_{n \in \Z^d} \F_{\R^d} (f)(n)e_n(x)=\sum_{m\in \Z^d} f(x+2\pi m), \quad x \in \R^d, 
\label{Poss}
\end{align}
\noi
for any function $f$ on $\R^d$
such that $|f(x)|\les \jb{x}^{-d-\dl}$ for some $\dl > 0$
and 
$\sum_{n \in \Z^d} |\F_{\R^d} (f)(n)| < \infty$.
The Poisson summation formula \eqref{Poss}
 is a typical tool 
 to pass information from $\R^d$ to a periodic torus $\T^d$;
 see~\cite{BO, OW, BOZ} for example.
 Here, $\F_{\R^d} (f)(n)$ denotes the Fourier transform of $f$ on $\R^d$ given by 
 \begin{align}
  \F_{\R^d} (f)(n) = \frac 1{(2\pi)^d} \int_{\R^d} f(x) e_{-n}(x) dx,
\label{FF1}
 \end{align}
 
 \noi
 where $dx = dx_{\R^d}$ is the standard Lebesgue measure on $\R^d$.
 Then, by applying \eqref{Poss} 
 (with a frequency truncation $\pi_N$ and taking $N \to \infty$) together with the asymptotics \eqref{G1}, we conclude that there exists
a smooth function $ R$ such that 
\begin{align}
G(x) = - c_d \log|x| +  R(x)
\label{G2}
\end{align}

\noi
for any $x \in \T^d \setminus \{0\}$.
See \cite[Section 2]{ORSW} for a related discussion.
Finally, from \eqref{IV2}, \eqref{G1a}, and~\eqref{G2}, 
we obtain
\begin{align*}
\E_\mu \big[u(x) u(y)\big] = G(x-y) = 
- c_d \log|x-y| +  R(x-y)
\end{align*}

\noi
for any $x, y \in \T^d$ with $x\ne y$.

\end{remark}

\subsection{Dynamical problems associated with 
the log-correlated Gibbs measures}
\label{SUBSEC:PDE}

From the viewpoint
of mathematical physics such as 
Euclidean quantum field theory, 
the construction of the Gibbs measures $\rho$
in \eqref{Gibbs1} is of interest in its own right.
In this subsection, we briefly discuss 
some examples of dynamical problems
associated with these log-correlated Gibbs measures.
These examples show the importance of studying the log-correlated
Gibbs measure $\rho$ in \eqref{Gibbs1}
from the (stochastic) PDE point of view.

The associated energy functional\footnote{Once again, we do not worry about renormalizations
in this formal discussion.}  
for the Gibbs measure  $\rho$ in \eqref{Gibbs1} 
is given by
\begin{align}
E(u)
= \frac 12 \int_{\T^d} |(1-\Dl)^\frac{d}{4} u|^2 dx 
- \frac \ld k \int_{\T^d}  u^{k} dx.
\label{Gibbs2}
\end{align}

\noi
The study of the Gibbs measures for Hamiltonian PDEs,  initiated by \cite{Fried, LRS, BO94, McKean, BO96},
has been an active field of research over the last decade. 
We first list examples of 
the Hamiltonian PDEs
generated by this energy functional $E(u)$ in \eqref{Gibbs2}
along with the references.

\smallskip

\begin{itemize}
\item[(i)] 
fractional nonlinear Schr\"odinger equation
(for complex-valued $u$):
\begin{align}
i\dt u + (1-\Dl)^{\frac d2}u -\ld | u |^{k-2}u=0.
\label{H1}
\end{align}

\noi
The equation \eqref{H1} corresponds to 
the nonlinear half-wave equation
(also known as  
the semi-relativistic nonlinear Schr\"odinger equation)
when $ d= 1$, 
to 
 the well-studied cubic nonlinear Schr\"odinger equation (NLS)
when $d = 2$ (\cite{BO96, OTh, DNY}), 
and to
the biharmonic nonlinear  Schr\"odinger equation
when $d = 4$.

In Appendix \ref{SEC:A}, we also provide a brief discussion
on the Gibbs measure 
for the Zakharov system when $d = 2$.

\smallskip

\item[(ii)] fractional nonlinear wave equation:\footnote{For \eqref{H2}, 
we need to add $\frac 12 \int_{\T^d} (\dt u)^2 dx$ to the energy functional $E(u)$ in \eqref{Gibbs2}.} 
\begin{align}
\dt^2 u + (1- \Dl)^\frac{d}{2}  u  - \ld  u^{k-1}  = 0.
\label{H2}
\end{align}

\noi
The equation \eqref{H2} corresponds to 
the nonlinear wave equation (NLW)
(or the nonlinear Klein-Gordon equation)
when $d = 2$ (\cite{OTh2}), 
and to the nonlinear beam equation 
when $d = 4$.

\smallskip

\item[(iii)] generalized Benjamin-Ono equation (with $d = 1$):\footnote{For \eqref{H3}, 
the coefficient of the potential energy in \eqref{Gibbs2} is slightly different.
Thanks to the conservation of the spatial mean $\int_\T u dx$
under the generalized Benjamin-Ono equation~\eqref{H3}, we can work on the mean-zero functions.
In this case, we consider the Gibbs measure associated
with the massless log-correlated Gaussian field
by replacing $(1-\dx^2)^\frac{1}{4}$ in \eqref{Gibbs2} with
 $(-\dx^2)^\frac{1}{4}$.}
\begin{align}  
\dt u+\mathcal{H}\dx^2u -\ld \dx (u^{k-1})=0,
\label{H3}
\end{align}

\noi
where $\mathcal H$ denotes the Hilbert transform
defined by $\ft {\mathcal H f}(n) = -i \text{sgn}(n) \ft f(n)$
with the understanding that $\ft {\mathcal H f}(0) = 0$.
The equation \eqref{H3} is known
as the Benjamin-Ono equation when $k = 3$ (\cite{TzBO, Deng})
and the modified Benjamin-Ono equation when $k = 4$.

\end{itemize}

\smallskip

\noi
Next, we list stochastic PDEs associated with 
the Gibbs measure $\rho$ in \eqref{Gibbs1}.

\smallskip

\begin{itemize}
\item[(iv)] parabolic stochastic quantization equation
\cite{PW}:
\begin{align}
 \dt u +  (1 -  \Dl)^\frac{d}{2}  u -\ld u^{k-1}  =  \sqrt 2\xi.
\label{H4}
\end{align}

\noi
Here, $\xi$ denotes the space-time white noise on $\T^d \times \R_+$.
When $d = 2$ and $\ld < 0$, \eqref{H4} 
corresponds to the standard parabolic $\Phi^k_2$-model
(\cite{DPD, RZZ, TsW}).

\smallskip

\item[(v)] canonical  stochastic quantization equation
\cite{RSS}:
\begin{align}
\dt^2 u + \dt u + (1 -  \Dl)^\frac{d}{2}  u -\ld u^{k-1}  = \sqrt{2} \xi.
\label{H5}
\end{align}

\noi
The equation \eqref{H5} corresponds to 
the stochastic damped nonlinear wave equation
when $d = 2$ (\cite{GKO, GKOT, Tolomeo3}), 
and to the stochastic damped nonlinear beam equation 
when $d = 4$.

\end{itemize}

\smallskip

\noi
When $d = 2$, 
the conservative  stochastic Cahn-Hilliard equation is known to 
(formally) preserve the Gibbs measure $\rho$ in \eqref{Gibbs1} (\cite{RYZ}).

For the equations listed above, 
once 
we establish  local well-posedness almost surely
with respect to the Gibbs measure initial data, 
Bourgain's invariant measure argument \cite{BO94, BO96}
allows us to construct almost sure global dynamics
and to prove invariance of the Gibbs measure.
However, 
since functions on the support
of the
log-correlated Gibbs measure~$\rho$ in \eqref{Gibbs1}
almost surely belong to  the $L^p$-based Sobolev spaces $W^{s, p}(\T^d)
\setminus L^p(\T^d)$ only for $ s < 0$ with any $1 \le p \le \infty$, 
there are only a handful of the well-posedness results
\cite{BO96, Deng, OTh2, GKO, DNY}
for the Hamiltonian PDEs mentioned above (including \eqref{H5}).

\begin{remark}\rm

We point out that as long as we can construct the Gibbs measure, 
a compactness argument 
with invariance of the truncated Gibbs measures 
and Skorokhod's theorem allows
us to construct (non-unique) global-in-time dynamics 
along with invariance of the Gibbs measure in some mild sense.
See \cite{AC, DPD2, BTTz, OTh, ORTh}.
In our current setting, this almost sure global existence
result holds
for (i) the defocusing case ($\ld < 0$ and even $k \geq 4$; see the discussion in Subsection~\ref{SUBSEC:1.1})
and (ii) the quadratic nonlinearity (i.e.~$k = 3$).
See Theorem \ref{THM:2} 
for the latter case.

\end{remark}

\begin{remark}\rm

Given $\dl > 0$, 
consider 
the  intermediate long wave equation (ILW) on $\T$: 
\begin{align}
\dt u -
 \Gdl \partial_x^2 u  - \dx (u^{2})    = 0, 
\label{ILW1}
\end{align}

\noi
where the dispersion operator $\Gdl $ is given by 
\begin{align}
\ft{\Gdl f}(n)  =  -  i \Big( \coth(\dl n )-\frac{1}{\dl  n} \Big) \ft f(n)\,, 
 \hspace{1mm} 
\quad n \in\Z.
\label{OP}
\end{align}

\noi
The equation \eqref{ILW1}
 models 
the internal wave propagation of the interface 
in a stratified fluid of finite depth $\dl > 0$, 
 providing 
 a natural connection between 
the  Benjamin-Ono regime ($\dl = \infty$) and 
the KdV regime ($\dl = 0$).
Indeed, there are results establishing convergence of 
ILW to the Benjamin-Ono equation (and the KdV equation)
as $\dl \to \infty$ (and $\dl \to 0$, respectively); see~\cite{Li, CFLOP, CLOP} and the references therein.
While it is not obvious from the rather complicated dispersive symbol in \eqref{OP}, 
the Gibbs measure associated to ILW is indeed log-correlated,
and the results in this paper apply 
to the Gibbs measure associated to the generalized ILW (where 
the nonlinearity $\dx (u^2)$ in \eqref{ILW1} is replaced by 
$\ld \dx (u^{k-1})$).   
Furthermore, as $\dl \to \infty$ (and $\dl\to0$), 
the Gibbs measure for the (generalized) ILW converges to that for 
the (generalized) Benjamin-Ono equation (and the (generalized) KdV equation, respectively)
in an appropriate sense.
See a recent work~\cite{LOZ} for a further discussion.
See also \cite{CLOZ}
for the construction and convergence of invariant measures
for ILW associated with higher order conservation laws.

\end{remark}

\subsection{Non-normalizability of the focusing Gibbs measure}
\label{SUBSEC:1.3}

We now turn our attention to  the focusing case.
In this subsection, we study the Gibbs measure $\rho$ in 
\eqref{Gibbs1}
with  the focusing quartic interaction
($\ld > 0$ and $k = 4$).
In this case, we prove the following 
 non-normalizability 
of the (renormalized) focusing Gibbs measure~$\rho$.

\begin{theorem}\label{THM:1}

Let  $\ld > 0$ and $k = 4$.  
Then,  given any $K > 0$,  we have
\begin{align}
\begin{split}
\sup_{N\in \N} Z_{K,N}\deff \sup_{N \in \N}  \E_{\mu} \Big[ \ind_{\{|\int_{\T^d} \, :u_N^2: \, dx| \, \leq K\}}  e^{ R_N (u)} \Big] = \infty        , 
\end{split}
\label{Gibbs4}
\end{align}

\noi
where $R_N$ is the renormalized potential energy defined in  \eqref{Wick0}
with $k = 4$. Moreover, the divergence rate of $Z_{K,N}$  is given by
\begin{align}
\log Z_{K,N} = \ld\frac{C_B}4 N^d \s_N^2 (1+o(1)) \sim N^d (\log N)^2, 
\label{divrate}
\end{align} 
as $N\to \infty$.
Here,  $C_B$ is the optimal constant in Bernstein's inequality\textup{:} 
\begin{equation} \label{bernstein}
\|P f\|_{L^4(\R^d, \frac{dx}{(2\pi)^d})}^4 \le C_B \|f\|_{L^2(\R^d, \frac{dx}{(2\pi)^d})}^4,
\end{equation}
where $P$ is the sharp Fourier projection 
onto the unit ball\textup{:}
$$ \widehat{Pf}(\xi) = \ind_{\{|\xi|\le 1\}} \widehat{f}(\xi), $$
and $\s_N$ is defined in \eqref{sigma1}.
Moreover, we have 
\begin{equation} \label{nonnormalisability}
Z_K \deff \E_{\mu} \Big[ \ind_{\{|\int_{\T^d} \, :u^2: \, dx| \, \leq K\}}  e^{ R (u)} \Big] = \infty, 
\end{equation}

\noi
where $R(u)$ is the limit of $R_N(u)$ defined in \eqref{conv1}.
 In particular, the focusing Gibbs measure \textup{(}even with a Wick-ordered $L^2$-cutoff\textup{)} 
can not be defined as a probability measure. 
	
\end{theorem}

When $d = 2$, 
Theorem \ref{THM:1} provides an alternative proof of the non-normalizability 
result
for 
of the focusing $\Phi^4_2$-measure  
due to  Brydges and Slade \cite{BS}
whose proof
 is based on analysis of a model closely related to 
the Berlin-Kac spherical model.
Furthermore, Theorem \ref{THM:1} provides a precise rate 
\eqref{divrate}
of divergence of 
the partition function $Z_{K, N}$.
Our strategy for proving the divergence rate \eqref{divrate}
is straightforward and thus is expected to be applicable to a wide range of models.

Our proof of Theorem \ref{THM:1} is based on the variational approach
due to Barashkov and Gubinelli~\cite{BG}.
More precisely, we will rely on 
the Bou\'e-Dupuis variational formula \cite{BD, Ust}; see Lemma~\ref{LEM:var3}.
Our main task is  to 
 construct a drift term  which  achieves the desired divergence~\eqref{Gibbs4}.
Our argument is inspired by   recent works 
by the third author with Weber~\cite{TW}
and by the first and third authors with Okamoto
\cite{OOT, OOTolo}.
In particular, our presentation closely follows but refines that in \cite{OOT}, 
where an analogous non-normalizability is shown
for focusing Gibbs measures on $\T^3$ with a quartic interaction of Hartree-type.
We point out that the argument in~\cite{OOT}
 shows  non-normalizability only for  large $K \gg 1$
 and thus we need to refine the argument to prove the divergence~\eqref{Gibbs4}
 for {\it any} $K > 0$. 
 The main new ingredient (as compared to \cite{OOT}) is  the construction a drift term which approximates a 
 blowup profile, such that the Wick-ordered $L^2$-cutoff does not exclude this blowup profile for \textit{any} cutoff size $K >0$. See, in particular,  Lemma \ref{LEM:leo2} and the proof of \eqref{pa5}. 
 We also mention  related works
\cite{LRS, BS, Rider, BB14, OST,  RSTW} on the non-normalizability (and other issues)
for focusing Gibbs measures.

\begin{remark}\rm
As a direct consequence of \eqref{Gibbs4}, we have
\begin{align*}
\sup_{N \in \N}  \E_{\mu} \Big[  e^{ R_N (u)} \Big]
& \ge
\sup_{N \in \N}  \E_{\mu} \Big[ \ind_{\{\int_{\T^d} \, :u_N^2: \, dx \, \leq K\}}  e^{ R_N (u)} \Big]\\
& \ge
\sup_{N \in \N}  \E_{\mu} \Big[ \ind_{\{|\int_{\T^d} \, :u_N^2: \, dx| \, \leq K\}}  e^{ R_N (u)} \Big] = \infty.
\end{align*}

\end{remark}

\begin{remark}\rm 

In the one-dimensional setting studied in \cite{LRS, OST}, 
the sharp Gagliardo-Nirenberg inequality on~$\R$ plays 
an important role in determining (non-)normalizability 
of the focusing Gibbs measure with a sextic interaction.
In our current problem with a quartic interaction, 
Bernstein's inequality \eqref{bernstein} on $\R^d$, 
which is essentially a frequency-localized version of Sobolev's inequality, 
plays a crucial role in determining the precise divergence rate \eqref{divrate}.
We point out that this particular form of Bernstein's inequality appears
due to the form of the regularization we use for our problem
(namely, the sharp frequency truncation onto the frequencies $\{|n|\le N\}$).
In the current singular setting where a renormalization is required, 
we need to start with a regularized problem.
However, 
there are different ways to regularize a problem, and different regularizations
lead to different divergence rates.
For example, if we instead use a smooth frequency truncation,
we would obtain a divergence rate with a different constant (while
the essential rate $N^d (\log N)^2$ in \eqref{divrate} remains the same).

\end{remark}

\begin{remark}\rm	
(i) An analogous non-normalizability result
holds for a focusing Gibbs measure with
the quartic interaction
even if we endow it with taming  by the Wick-ordered $L^2$-norm.
See Remark \ref{REM:NLW}.

\smallskip

\noi
(ii)
By controlling combinatorial complexity, 
we can extend the non-normalizability result in Theorem \ref{THM:1}
to the higher order interactions $k \geq 5$ in the focusing case
(i.e.~either $k$ is odd or $\ld > 0$ when $k$ is even).

\smallskip

\noi
(iii)
In terms of dynamical problems, 
Theorem \ref{THM:1} states
that Gibbs measures
associated with 
the equations 
listed  in Subsection \ref{SUBSEC:PDE}
do not exist for (i) $\ld> 0$ and $k \ge 4$
or (ii) odd $k \ge 5$.
This list in particular includes 

\smallskip

\begin{itemize}
\item
the focusing $L^2$-(super)critical fractional NLS \eqref{H1}
(including the focusing (super)cubic NLS on $\T^2$), 

\smallskip
\item 
the focusing $L^2$-(super)critical fractional NLW \eqref{H2}
(including the focusing (super)cubic NLW on $\T^2$
and  the focusing (super)cubic nonlinear beam equation on $\T^4$),

\smallskip
\item 
the focusing modified Benjamin-Ono equation \eqref{H3}
(and  
the focusing generalized Benjamin-Ono equation with $k \ge 5$).

\smallskip

\end{itemize}
See also Appendix \ref{SEC:A} for a brief discussion
on the two-dimensional Zakharov system.

\end{remark}

\begin{remark}\label{REM:phase}\rm
In a recent work \cite{OST}, the first and third authors with Okamoto 
studied the construction of the $\Phi^3_3$-measure on $\T^3$
(i.e.~\eqref{Gibbs1} with $d = 3$ and $k = 3$), and 
established 
the following phase transition: normalizability in the weakly nonlinear regime ($|\ld|\ll 1$)
and 
 non-normalizability in the strongly nonlinear regime ($|\ld|\gg 1$), 
 where the latter result was obtained based on the strategy in the current paper.
In particular, in view of the non-normalizability of the $\Phi^3_3$-measure
in the strongly nonlinear regime, 
we expect that the same approach would 
yield non-normalizability 
of the focusing $\Phi^k_3$-measure 
for $k \ge 4$ (namely, (i) for even $k \ge 4$ with $\ld > 0$
or (ii) for odd $k \ge 5$ with $\ld \ne 0$).
\end{remark}

\subsection{Gibbs measure with the cubic interaction}
\label{SUBSEC:1.4}

Let us first go over the focusing Gibbs measure construction in the two-dimensional setting.
In \cite{BO95}, 
Bourgain reported Jaffe's construction of a  $\Phi^3_2$-measure endowed with a Wick-ordered
$ L^2$-cutoff:
\begin{align}
d\rho(u) = Z^{-1}
\ind_{\{\int_{\T^2} :\,u^2: \, dx\, \leq K\}} 
e^{ \int_{\T^2}  :u^3: \, dx  }d \mu(u) .
\label{Gibbs5}
\end{align}

\noi 
Note that 
the measure in \eqref{Gibbs5}
is not suitable to generate any NLS\,/\,NLW\,/\,heat dynamics since
(i)~the renormalized cubic power $:\!u^3 \!:$ makes sense only 
in the real-valued setting and hence is not suitable for the Schr\"odinger equation
 and (ii) NLW and the heat equation do not preserve the $L^2$-norm of a solution
 and thus are incompatible with the Wick-ordered $L^2$-cutoff.
In \cite{BO95}, 
Bourgain instead proposed to consider the  Gibbs measure
of the form:\footnote{The choice of the exponent $\g=2$ in $A\big(\int_{\T^2} :u^2:  dx\big)^{\g}$ (with $A\gg 1$) is optimal. See Remark \ref{REM:exp}}.
\begin{align}
d\rho(u ) = Z^{-1}
e^{ \int_{\T^2}  :u^3: \, dx  - A
\big(\int_{\T^2} :\,u^2: \, dx\big)^2} d \mu( u ) 
\label{Gibbs6}
\end{align}

\noi
(for sufficiently large $A>0$)
in studying NLW dynamics on $\T^2$.\footnote{For the NLW dynamics, 
we need to couple $\rho$ on the $u$-component  with the white noise measure $\mu_0$ on the $\dt u$-component (which is independent from $\rho$). More precisely, the Gibbs measure is of the form $\rhoo=\rho  \otimes \mu_0$, 
where the $\Phi^3_2$-measure $\rho$ in \eqref{Gibbs6} is on the $u$-component and the white noise measure  $\mu_0$ is on the $\dt u$-component.}

We now extend the construction of the Gibbs measures in \eqref{Gibbs5}
and \eqref{Gibbs6} to a general dimension $d \geq 1$.
Given $N \in \N$, let 
\begin{align}
\begin{split}
R_N^\diamond (u)
&=  \frac \ld 3\int_{\T^d}  :\! u_N^3 \!:  dx
- A \, \bigg( \int_{\T^d} :\! u_N^2 \!: dx\bigg)^2, 
\end{split}
\label{K2}
\end{align}

\noi
where the coupling constant
$\ld \in \R\setminus  \{0\} $ denotes the strength of cubic   interaction, 
and 
define the truncated renormalized Gibbs measure $\rho_{N}$ by 
\begin{align}
d \rho_{N} (u) = Z_{N}^{-1} e^{R_N^\diamond(u)} d \mu(u).
\label{GibbsN1}
\end{align}

\noi
Then, we have the following result 
for  the focusing Gibbs measure with a cubic interaction.

\begin{theorem} \label{THM:2}
Let  $\ld \in \R\setminus\{0\}$.
 Given any finite $ p \ge 1$, 
  there exists sufficiently large $A = A(\ld, p) > 0$ 
such that 
$R_N^\diamond $  in \eqref{K2}
converges to some limit $R^\diamond$
in $L^p(\mu)$.
Moreover, 
there exists $C_{p, d, A} > 0$ such that 
\begin{equation}
\sup_{N\in \N} \Big\| e^{R_N^\diamond(u)}\Big\|_{L^p(\mu)}
\leq C_{p,d, A}  < \infty.
\label{exp3}
\end{equation}

\noi
In particular,  we have
\begin{equation}
\lim_{N\rightarrow\infty}e^{ R_N^\diamond(u)}=e^{R^\diamond(u)}
\qquad \text{in } L^p(\mu).
\label{exp4}
\end{equation}

\noi
As a consequence, 
the truncated renormalized Gibbs measure $\rho_{N}$ in \eqref{GibbsN1} converges, in the sense of \eqref{exp4}, 
to the focusing Gibbs measure $\rho$ given by
\begin{align*}
d\rho(u)= Z^{-1} e^{R^\diamond(u)}d\mu(u).
\end{align*}
	
\noi
Furthermore, 
the resulting Gibbs measure $\rho$ is equivalent 
to the log-correlated Gaussian  field~$\mu$.
	
\end{theorem}

As for the convergence of $R_N^\diamond$, we omit details since
the argument is standard.
See,  for example,  
\cite[Proposition 1.1]{OTh}, 
\cite[Proposition 3.1]{OTz}, 
 \cite[Lemma 4.1]{GOTW}, 
and \cite[Lemma  5.1]{OOT} for related details.
As mentioned in Subsection \ref{SUBSEC:1.1}, 
the main task is to prove 
the uniform integrability bound \eqref{exp3}.
Once this is done, the rest follows from a standard argument.
In Section \ref{SEC:4}, 
we establish the bound \eqref{exp3}
by using  the variational formulation.

\begin{remark} \label{REM:cub}\rm

Note that
\begin{align}
 \ind_{\{|\,\cdot \,| \le K\}}(x) \le \exp\big( -  A |x|^\gamma\big) \exp(A K^\g)
\label{H6}
\end{align}

\noi
for any $K, A , \g > 0$.
Then, 
 the following
uniform bound
for  the focusing cubic interaction:
\begin{equation*}
\sup_{N\in \N} \Big\|
\ind_{\{|\int_{\T^d} :\,u^2: \, dx| \leq K\}} 
 e^{R_N(u)}\Big\|_{L^p( \mu)}
\leq C_{p,d, K}  < \infty
\end{equation*}

\noi
for any $K > 0$
follows as a direct consequence
of the uniform bound \eqref{exp3} and \eqref{H6}
with $\g = 2$, 
where 
 $R_N$ is as in \eqref{Wick0} with $\ld \in \R\setminus\{ 0\}$ and $k = 3$.
This allows us to construct the log-correlated  Gibbs measure  
with the cubic interaction (with a Wick-ordered
$ L^2$-cutoff):
\begin{align*}
d\rho(u) = Z^{-1}
\ind_{\{|\int_{\T^d} :\,u^2: \, dx| \leq K\}} 
e^{ \frac \ld 3\int_{\T^d}  :u^3: \, dx  }d \mu(u) 
\end{align*}

\noi 
as a limit of its truncated version
(for any $\ld \in \R\setminus\{0\}$ and $ K > 0$).

\end{remark}

\begin{remark} \rm
 In \cite{TzBO}, Tzvetkov constructed the Gibbs
measure (with a Wick-ordered
$ L^2$-cutoff) for the  Benjamin-Ono equation \eqref{H3} with $k = 3$.
Theorem \ref{THM:2} and Remark \ref{REM:cub} provide
an alternative proof of the construction of the Gibbs
measure 
for the  Benjamin-Ono equation.
\end{remark}

\begin{remark}\label{REM:NLW} \rm
(i) It follows from  Theorem \ref{THM:1} and \eqref{H6} that 
an analogue of Theorem \ref{THM:2} fails
for the quartic interaction ($k = 4$). 
More precisely, we have
\begin{equation*}
\sup_{N\in \N} \bigg\| 
\exp \bigg(
\frac \ld 4 \int_{\T^d}  :\! u_N^4 \!:  dx
- A \, \Big| \int_{\T^d} :\! u_N^2 \!: dx\Big|^\g
\bigg)\bigg\|_{L^p(\mu)}
=  \infty
\end{equation*}
	
\noi
for any $\ld, A, \g > 0$.

\smallskip

\noi
(ii)
If we consider 
a smoother base Gaussian 
measure  $\mu_\al$, then 
we can prove the following uniform exponential integrability bound;
given any $\ld > 0$, $\al > \frac d2$, 
and finite $p \geq 1$, 
there exists
sufficiently large  $A = A(\ld, \al, p) > 0$
and $\g = \g(\al) > 0$
such that 
\begin{equation}
\sup_{N\in \N} \bigg\| 
\exp \bigg(
\frac \ld 4 \int_{\T^d}   u_N^4   dx
- A \, \Big( \int_{\T^d}  u_N^2  dx\Big)^\g 
\bigg)\bigg\|_{L^p(\mu_\al)}
\leq C_{p, d, A} < \infty.
\label{Q1}
\end{equation}
	
\noi
Here,    $\mu_\al$  denotes  the Gaussian measure
 with a formal density 
\begin{align}
d \mu_\al 
= Z^{-1} e^{-\frac 12 \| u\|_{H^{\al} }^2    } du.
\label{gauss1}
\end{align}

\noi
See Appendix \ref{SEC:B}
for the proof of \eqref{Q1}.
The bound \eqref{Q1}  allows us to construct 
the focusing Gibbs measure 
with a focusing quartic interaction
of the form:
\begin{align}
d\rho_\al = Z^{-1}
e^{ \frac \ld 4 \int_{\T^d}  u^4  dx  - A
\big(\int_{\T^d} u^2  dx\big)^\g} d \mu_\al .
\label{Q2}
\end{align}

\noi
Moreover, in view of \eqref{H6}, 
we can also construct the following 
 focusing Gibbs measure with an $L^2$-cutoff:
\begin{align}
d\rho_\al = Z^{-1}
\ind_{\{\int_{\T^d} |u|^2 dx \leq K\}} 
e^{ \frac \ld 4 \int_{\T^d}  |u|^4  dx 
} d \mu_\al 
\label{Q3}
\end{align}

\noi
for any $K > 0$.

In  \cite{STz1, STz2}, 
Sun and Tzvetkov recently studied the following 
fractional NLS on $\T$:
\begin{align}
i\dt u + (1-\dx^2)^{\al}u -\ld | u |^{2}u=0
\label{Q4}
\end{align}

\noi
in the defocusing case ($\ld < 0$).
They
 proved almost sure local well-posedness
of \eqref{Q4} with respect to the Gaussian measure $\mu_\al$
in \eqref{gauss1}
for $\al > \frac {31- \sqrt{233}}{28} \approx 0.562 \ (  \, > \frac 12)$,\footnote{See also a recent preprint \cite{LW}, 
where the authors covered the range $\al > \frac 12$.} 
which in turn yielded almost sure global well-posedness
with respect to the defocusing Gibbs measure
(namely, $\rho_\al$ in~\eqref{Q3} without an $L^2$-cutoff)
and invariance of the defocusing Gibbs measure.
Since their local result also holds
in the focusing case ($\ld>0$), 
our construction of the focusing Gibbs measure 
$\rho_\al$ in \eqref{Q3}
implies 
 almost sure global well-posedness of \eqref{Q4}
with respect to the focusing Gibbs measure  $\rho_\al$ in~\eqref{Q3} 
and its invariance under the dynamics of \eqref{Q4}
for the same range of $\al$.

\smallskip

\noi
(iii)
Theorem \ref{THM:1} and Part (ii) of this remark show that 
in the case of the focusing quartic interaction, 
there is no phase transition,  depending on the value of $\ld > 0$.
Compare this with the situation in \cite{OOT, OOTolo}, 
where such a phase transition
(as described in Remark \ref{REM:phase})
 was established in the critical case.
It may be of interest to pursue 
the issue of a possible phase transition
for a higher order focusing interaction, in the non-singular regime $\al > \frac d2$.

\end{remark}

\section{Preliminary lemmas}
\label{SEC:2}

In this section, we recall  basic definitions and lemmas  used in this paper.

Let $s \in \R$ and $1 \leq p \leq \infty$.
We define the $L^2$-based Sobolev space $H^s(\T^d)$
by the norm:
\begin{align*}
\| f \|_{H^s} = \| \jb{n}^s \ft f (n) \|_{\l^2_n}.
\end{align*}

\noi
We also define the $L^p$-based Sobolev space $W^{s, p}(\T^d)$
by the norm:
\begin{align*}
\| f \|_{W^{s, p}} = \big\| \F^{-1} [\jb{n}^s \ft f(n)] \big\|_{L^p}.
\end{align*}

\noi
When $p = 2$, we have $H^s(\T^d) = W^{s, 2}(\T^d)$.

\subsection{Deterministic estimates}
We first recall the following interpolation
and fractional Leibniz rule.
As for the second estimate \eqref{bilinear+}, 
see \cite[Lemma 3.4]{GKO}.

\begin{lemma}\label{LEM:prod}
The following estimates hold.

\noi
\textup{(i) (interpolation)} 
For  $0 < s_1  < s_2$, we have
\begin{equation*}
\| u \|_{H^{s_1}} \le \| u \|_{H^{s_2}}^{\frac{s_1}{s_2}} \| u \|_{L^2}^{\frac{s_2-s_1}{s_2}}.
\end{equation*}

\smallskip
	
\noi 
\textup{(ii) (fractional Leibniz rule)} Let $0\le s \le 1$. Suppose that 
$1<p_j,q_j,r < \infty$, $\frac1{p_j} + \frac1{q_j}= \frac1r$, $j = 1, 2$. 
Then, we have\footnote{We use the convention that the symbol $\lesssim$ indicates that inessential constants are suppressed in the inequality.}  
\begin{equation}  
\| \jb{\nb}^s (fg) \|_{L^r(\T^d)} 
\les \Big( \| f \|_{L^{p_1}(\T^d)} 
\| \jb{\nb}^s g \|_{L^{q_1}(\T^d)} + \| \jb{\nb}^s f \|_{L^{p_2}(\T^d)} 
\|  g \|_{L^{q_2}(\T^d)}\Big),
\label{bilinear+}
\end{equation}

\noi
where $\jb{\nb} = \sqrt{1 - \Dl}$.

\end{lemma}

The next lemma states almost optimal Bernstein's inequality on $\T^d$.

\begin{lemma}\label{LEM:Bernstein}
Given $N \in \N$, let $\pi_N$ be the frequency projector as in \eqref{pi}.
Then, we have 
\begin{equation*}
\|\pi_N f\|_{L^4(\T^d)}^4 \le C_B N^d (1 +o(1)) \| f\|_{L^2(\T^d)}^4
\end{equation*}

\noi
as $N \to \infty$, 
where $C_B$ is the optimal constant for Bernstein's inequality \eqref{bernstein} on $\R^d$.
\end{lemma}

\begin{proof}

Given $N \in \N$, 
let $C_{B,N}$ be the optimal constant for the following inequality on $\T^d$:
\begin{equation}\label{bernstein2N}
\|\pi_N f\|_{L^4(\T^d)}^4 \le C_{B,N} N^d \| \pi_N f\|_{L^2(\T^d)}^4,
\end{equation}
and let $f_N$ be an optimizer for \eqref{bernstein2N} with $\|f_N\|_{L^2(\T^d)} = 1$ and $\pi_N f_N = f_N$.
In particular, we have
\begin{align}
 \|f_N \|_{L^4(\T^d)}^4 = C_{B,N} N^d. 
 \label{bern3}
\end{align}

\noi
Note that such an optimizer exists, since the set 
$\{f_N: \|f_N\|_{L^2(\T^d)} = 1, \, \pi_N f_N = f_N\}$ is compact. 
Moreover, by Sobolev's inequality on the torus, we have 
\begin{align}
C_{B,N} \les 1, 
\label{bern33}
\end{align}

\noi
 uniformly in $N \in \N$. 
Then, in view of \eqref{bern3},  it suffices to show that 
\begin{align}
\limsup_{N\to \infty} N^{-d}\|f_N\|_{L^4(\T^d)}^4 \le C_B. 
\label{tsrf0}
\end{align}

Fix small $\eps> 0$.
Let $ \chi_\eps  \in C^\infty_c(\R^d; [0, 1])$ be a smooth bump function which is compactly supported 
on $[-\pi,\pi)^d\cong \T^d$
 such that 
$\chi_\eps \equiv 1$ on $[-\pi+c_0 \eps,\pi-c_0\eps]^d$
for some small $c_0 = c_0 > 0$ to be chosen later.
Recalling that $dx_{\T^d} = (2\pi)^{-d} dx$ is the normalized Lebesgue measure on $\T^d$, 
we see that $\|f_N\|_{L^4(\T^d)}^4$ is the average of $|f(x)|^4$ on $\T^d$.
Hence, by suitably translating $f_N$
(that does not affect its optimality)
 and choosing 
$c_0 = c_0 > 0$ sufficiently small (independent of small $\eps > 0$ and $N \in \N$), we have 
\begin{align}
 \|\chi_\eps^2f_N\|_{L^4(\T^d)} \ge (1-\eps) \|f_N\|_{L^4(\T^d)},
 \label{bern3a}
\end{align}

\noi
uniformly in $N \in \N$.
In the following, when we view $f_N$ as a function on $\R^d$, 
we simply view it as a periodic function:  $f(x) = f(x + 2\pi m)$, $m \in \Z^d$.

Let $\ta \in 
 C^\infty_c(\R^d; [0, 1])$
 be a smooth radial bump function on $\R^d$ such that 
$\ta (\xi) = 1$ for $|\xi|\le 1$
and 
$\ta (\xi) = 0$ for $|\xi| > 2$.
Given $M  > 0$, set $\ta_M(\xi) = \ta\big(\frac \xi M\big)$.
Now, we 
set  
\begin{align}
\chi_{\eps,M} = \Q_M(\chi_\eps) : = \F^{-1}_{\R^d}(\ta_M)* \chi_{\eps},
\label{bernx}
\end{align}

\noi 
where $\F^{-1}_{\R^d}$ is the inverse Fourier transform on $\R^d$.
Namely, 
$\chi_{\eps,M}$ is the frequency-localized  version 
of $\chi_\eps$
onto the frequencies $\{\xi \in \R^d: |\xi|\le 2 M\}$.
Then, by choosing 
$M  = M(\eps, N) > 0$ sufficiently large, we have
\begin{align}
 \|\chi_\eps -  \chi_{\eps, M}\|_{L^1(\R^d)\cap L^\infty(\R^d)} 
= \|(\Id - \Q_M)\chi_\eps\|_{L^1(\R^d)\cap L^\infty(\R^d)}
\ll \eps N^{-\frac d4}. 
\label{bern3b}
\end{align}
Since $\chi_\eps$ is a Schwartz function, 
we have  $M(\eps,N) = o(N)$
 for each fixed $\eps> 0$.

By  the definition \eqref{bernx} of $\chi_{\eps, M}$
and choosing $M = M(\eps, N) = o(N)$ possibly larger, 
we have 
\begin{align}
\begin{split}
\| \chi_{\eps,M}(\,\cdot  + 2\pi m)
\|_{L^\infty([-\pi,\pi)^d)}  
&  = \sup_{x \in [-\pi,\pi)^d}
M^{d} \int_{\R^d}\chi_\eps(x+2\pi m - y) \F_{\R^d}^{-1}(\ta)(M y) dy \\
& \les \frac {M^d }{\jb{Mm}^{2d+1} }
\ll \frac{\eps}{\jb{m}^{d}}, 
\end{split}
\label{bern3c}
\end{align}

\noi
uniformly in $m \in \Z^d\setminus\{0\}$, 
where the penultimate step follows from 
$\supp \chi_\eps \subset [- \pi, \pi)^d$
and the fact that $\ta$ is a Schwartz function.
Then, 
from the periodicity of $f_N$, 
$\supp \chi_\eps \subset [- \pi, \pi)^d$, 
\eqref{bern3b}, and \eqref{bern3c}, we obtain
\begin{align*}
 & \| (\chi_{\eps} - \chi_{\eps,M})f_N \|_{L^4(\R^d, \frac{dx}{(2\pi)^d})}\\
&   = \bigg(\frac{1}{(2\pi)^d}\sum_{m \in \Z^d} 
\int_{[-\pi,\pi)^d} (\chi_{\eps} - \chi_{\eps,M})^4(x + 2\pi m)
| f_N(x)|^4 dx\bigg)^\frac 14\\
&  \le
 \| f_N \|_{L^4(\T^d)}
 \bigg( \|\chi_\eps -  \chi_{\eps, M}\|_{L^\infty([-\pi,\pi)^d)}^4
+   \sum_{m \in \Z^d\setminus\{0\}} 
\| \chi_{\eps,M}^4(\,\cdot  + 2\pi m)
\|_{L^\infty([-\pi,\pi)^d)}^4 \bigg)^\frac 14
\\
& \ll \eps  \| f_N \|_{L^4(\T^d)}.
\end{align*}

\noi
As a consequence, we have 
\begin{align}
\begin{split}
&  \| (\chi_{\eps}^2  - \chi_{\eps,M}^2  )f_N \|_{L^4(\R^d, \frac{dx}{(2\pi)^d})\cap L^2(\R^d, \frac{dx}{(2\pi)^d})}\\
& \le  \| \chi_{\eps} - \chi_{\eps,M} \|_{L^\infty(\R^d)\cap L^4(\R^d, \frac{dx}{(2\pi)^d}      )}\\
& \quad \times \bigg(2 \| \chi_{\eps} f_N \|_{L^4(\R^d, \frac{dx}{(2\pi)^d})}+ 
\| (\chi_{\eps} -  \chi_{\eps,M})f_N \|_{L^4(\R^d, \frac{dx}{(2\pi)^d})}\bigg)\\
& \ll \eps N^{-\frac d4} \| f_N \|_{L^4(\T^d)}.
\end{split}
\label{bern3d}
\end{align}

\noi
Hence, from~\eqref{bern3a} and \eqref{bern3d}, 
we have
\begin{align}
(1-2\eps)\|f_N\|_{L^4(\T^d)} \le  \| \chi_{\eps,M}^2f_N \|_{L^4(\R^d, \frac{dx}{(2\pi)^d})}
\le (1+\eps)\|f_N\|_{L^4(\T^d)} 
\label{bern4}
\end{align}

\noi
for any small $\eps > 0$, 
uniformly in $N \in \N$.


Define the function $g_N, g_{N,M}:\R^d \to \C$ by setting
\begin{align}
 g_N(x) = \frac 1{N^{\frac d2}} \chi_\eps^2\Big(\frac xN\Big)f_N\Big(\frac x N\Big)
 \qquad \text{and} \qquad g_{N,M}(x) = \frac 1{N^{\frac d2}}
\chi_{\eps,M}^2\Big(\frac xN\Big)
  f_N\Big(\frac x N\Big).
\label{bern5}
\end{align} 

\noi
Then, from \eqref{bern4} and \eqref{bern5},  we have 
\begin{align}
N^\frac d4
\|g_{N,M}\|_{L^4(\R^d, \frac{dx}{(2\pi)^d})} 
= \| \chi_{\eps,M}^2f_N \|_{L^4(\R^d, \frac{dx}{(2\pi)^d})} 
\ge (1-2\eps)\|f_N\|_{L^4(\T^d)}.
\label{trsf1}
\end{align}

\noi
By H\"older's inequality and \eqref{bern3d} with \eqref{bern3} and \eqref{bern33}, we have 
\begin{align*}
\|g_N - g_{N,M}\|_{L^2(\R^d, \frac{dx}{(2\pi)^d})}
 &= \| (\chi_\eps^2 - \chi_{\eps,M}^2)f_N \|_{L^2(\R^d, \frac{dx}{(2\pi)^d})} \\
& \ll \eps.
\end{align*}

\noi
Noting that $\|g_N\|_{L^2(\R^d, \frac{dx}{(2\pi)^d})} = \|\chi_\eps^2f_N \|_{L^2(\T^d)} \le 1$, we then obtain 
\begin{align}
\|g_{N,M}\|_{L^2(\R^d, \frac{dx}{(2\pi)^d})} \le (1+\eps). 
\label{trsf2}
\end{align}

Finally,  recalling that the Fourier support of $f_N = \pi_N f $ (as a function on $\T^d$) is contained in $\{n \in \Z^d |n| \le N\}$ and the Fourier support of $\chi_{\eps,M}$ (as a function on $\R^d$ is contained in 
$\{\xi \in \R^d: |\xi| \le 2M\}$
and that  $M(\eps,N) = o(N)$, it follows from \eqref{bern5} that 
\begin{align}
 \supp(\ft g_{N,M}) \subset \bigg\{ \xi \in \R^d : |\xi| \le \frac{N+2M}{N}\bigg\} 
 \subset  \big\{ \xi \in \R^d :  |\xi| \le 1 + o(1) \big\}.
 \label{bern6}
\end{align}
Therefore, from  \eqref{trsf1} and 
(the scaled version of) \eqref{bernstein} with \eqref{bern6} followed by \eqref{trsf2}, we 
conclude that 
\begin{align*}
N^{-d}\|f_N\|_{L^4(\T^d)}^4 
& \le (1 -2 \eps)^{-4}  \|g_{N,M}\|_{L^4(\R^d, \frac{dx}{(2\pi)^d})}^4 \\
& \le (1-2\eps)^{-4}4 C_B  (1 +o(1)) \|g_{N,M}\|_{L^2(\R^d, \frac{dx}{(2\pi)^d})}^4 \\
& \le \bigg(\frac{1+\eps}{1-2\eps}\bigg)^4 C_B  (1 +o(1)).
\end{align*}
Since $\eps>0 $ is arbitrary, by taking the $\limsup$ as $N\to \infty$, we obtain \eqref{tsrf0}.
\end{proof}

\subsection{Tools from stochastic analysis}

Next, we recall the Wiener chaos estimate (Lemma~\ref{LEM:hyp}).
For this purpose, we first recall 
basic definitions
from stochastic analysis;
see \cite{Bog, Shige}. 
Let $(H, B, \nu)$ be an abstract Wiener space.
Namely, $\nu$ is a Gaussian measure on a separable Banach space $B$
with $H \subset B$ as its Cameron-Martin space.
Given  a complete orthonormal system $\{e_j \}_{ j \in \N} \subset B^*$ of $H^* = H$, 
we  define a polynomial chaos of order
$k$ to be an element of the form $\prod_{j = 1}^\infty H_{k_j}(\jb{x, e_j})$, 
where $x \in B$, $k_j \ne 0$ for only finitely many $j$'s, $k= \sum_{j = 1}^\infty k_j$, 
$H_{k_j}$ is the Hermite polynomial of degree $k_j$ as in \eqref{Herm0}, 
and $\jb{\cdot, \cdot} = \vphantom{|}_B \jb{\cdot, \cdot}_{B^*}$ denotes the $B$-$B^*$ duality pairing.
We then 
denote the closure  of 
polynomial chaoses of order $k$ 
under $L^2(B, \nu)$ by $\mathcal{H}_k$.
The elements in $\mathcal{H}_k$ 
are called homogeneous Wiener chaoses of order $k$.
We also set
\begin{align}
\mathcal{H}_{\leq k} = \bigoplus_{j = 0}^k \mathcal{H}_j
\notag
\end{align}

\noi
for $k \in \N$.

Let $L = \Dl -x \cdot \nabla$ be 
the Ornstein-Uhlenbeck operator.\footnote{For simplicity, 
we write the definition of the Ornstein-Uhlenbeck operator $L$
when $B = \R^d$.}
Then, 
it is known that 
any element in $\mathcal H_k$ 
is an eigenfunction of $L$ with eigenvalue $-k$.
Then, as a consequence
of the  hypercontractivity of the Ornstein-Uhlenbeck
semigroup $U(t) = e^{tL}$ due to Nelson \cite{Nelson}, 
we have the following Wiener chaos estimate
\cite[Theorem~I.22]{Simon}.

\begin{lemma}\label{LEM:hyp}
Let $k \in \N$.
Then, we have
\begin{equation*}
\|X \|_{L^p(\O)} \leq (p-1)^\frac{k}{2} \|X\|_{L^2(\O)}
\end{equation*}
	
\noi
for any $p \geq 2$
and any $X \in \mathcal{H}_{\leq k}$.
	
\end{lemma}

\begin{lemma}\label{LEM:boundarymeasure}
Let $\nu_N$ be the law of $I_N \deff \int_{\T^d}:\!  u_N^2(x)  \!: dx$, 
where $u$ is as in \eqref{IV2}
and $u_N = \pi_N u$. Then,  for every $N \in \N$, 
$\nu_N$ is absolutely continuous with respect to the Lebesgue measure
$\ld$ on $\R$. 
Moreover, we have
\begin{align}
\bigg\|\frac{d\nu_N}{d\ld} \bigg\|_{L^\infty(\R)} \les 1, 
\label{PP1}
\end{align}
uniformly in $N \in \N$. 
As a consequence, we have
\begin{align}
\mu \bigg(\int_{\T^d}:\!  u^2(x)    \!: dx = K\bigg) = 0
\label{PP2}
\end{align}

\noi
for any $K \in \R$, 
where $\mu$ is the log-correlated Gaussian free field 
defined in  \eqref{gauss0}.

\end{lemma}

\begin{proof}
By the definition \eqref{Wick1} of $:\!  u_N^2 \!: \,$
with \eqref{IV2}, we have 
\begin{align*}
 \int_{\T^d} :\!  u_N^2(x) \!:  dx 
 & = \sum_{0 \le |n|\le N} \frac{|g_n|^2 - 1}{\jb{n}^d} \\
 & =  \sum_{0\le|n|\le 1} \frac{|g_n|^2-1}{\jb{n}^d}+ 
 \sum_{2\le|n|\le N} \frac{|g_n|^2-1}{\jb{n}^d} \\& =: A_1 + A_{2, N}
\end{align*}

\noi
with the understanding that $A_{2, N} = 0$ when $N = 1$.
Because of independence of the Gaussians $\{g_n\}_{|n|>2}$ from $g_0$ and $g_1$, 
the random variables $A_1$ and $A_{2, N}$ are independent. 
Note that
 the law $\mu_1$ of $A_1$ 
(and $\mu_{2, N}$ of $A_{2, N}$ when $N \ge 2$, respectively) is absolutely continuous
with respect to the Lebesgue measure $\ld$ on $\R$.
Thus, we have  $d \mu_1 = \s_1 d\ld$ 
for some $\s_1 \in L^1(\R)$
(and  $d \mu_{2, N} = \s_{2, N} d\ld$ 
for some $ \s_{2, N} \in L^1(\R)$
when $N \ge 2$, respectively).

We have
\begin{align*}
 \sum_{0\le|n|\le 1} \frac{|g_n|^2-1}{\jb{n}^d}
 = (g_0^2- 1) + 2^{1-\frac d2}(|g_1|^2 -1).
\end{align*}

\noi
Letting 
 $\s_{10}$ (and $\s_{11}$) be the density
 for $g_0^2- 1$ (and 
$ 2^{1-\frac d2}(|g_1|^2 -1)$, respectively), 
we have \[\s_1 = \s_{10} * \s_{11}.\]

\noi
Note that  $g_0^2$ is a chi-square distribution of one degree of freedom
and thus the density $\s_{10}$ for $g_0^2- 1$ is unbounded.\footnote{In particular, 
\eqref{PP1} is false when $N = 0$.}
On the other hand, 
$2|g_1|^2 = 2(\Re g_1)^2 + 2(\Im g_1)^2$
is a chi-square distribution of two degrees of freedom
and thus the density $\s_{11}$ for $ 2^{1-\frac d2}(|g_1|^2 -1)$ is bounded.
Hence, by Young's inequality, we have
\begin{align*}
\|\s_1\|_{L^\infty(\R)}
= \|\s_{10}* \s_{11}\|_{L^\infty(\R)}
\le \|\s_{10}\|_{L^1(\R)}
 \| \s_{11}\|_{L^\infty(\R)} <\infty, 
\end{align*}

\noi
which proves \eqref{PP1}, when $N = 1$.
Next, we consider the case $N \ge 2$.
Denoting by 
 $\s_{2n}$  the density
 for 
$ 2\jb{n}^{-d}(|g_n|^2 -1)$, 
by Young's inequality, we have 
\begin{align}
\|\s_{2, N}\|_{L^1(\R)} 
= 
\| \s_{22} * \s_{23} * \cdots \s_{2N}\|_{L^1(\R)} 
\le \prod_{n = 2}^N \|\s_{2n}\|_{L^1(\R)} 
= 1, 
\label{XY1}
\end{align}

\noi
where the last equality holds
since $\s_{2n}$ is a density of a probability measure.
Hence, 
by Young's inequality with \eqref{XY1}, we have 
\begin{align*}
\bigg\|\frac{d\nu_N}{d\lambda}\bigg\|_{L^\infty(\R)} 
&= \bigg\|\frac{d\Law( A_1 + A_{2, N})}{d\lambda}\bigg\|_{L^\infty(\R)} \\
&= \| \s_1 \ast \s_{2, N}\|_{L^\infty(\R)} \\
&\le \| \s_1 \|_{L^\infty(\R)} \|\s_{2, N}\|_{L^1(\R)} \\
&= \| \s_1 \|_{L^\infty(\R)} \\
&\les 1,
\end{align*}

\noi
uniformly in $N \ge 2$.
This proves \eqref{PP1}.

Let $I_\infty = \int_{\T^d}:\!  u^2(x)  \!: dx$.
Since $I_N$ converges to $I_\infty$ in law (see, for example, \cite[Proposition~1.1]{OTh}), 
it follows from the  Portmanteau theorem and \eqref{PP1} that 
\begin{align*}
\PP(I_\infty = K) 
& \le \PP\big(I_\infty \in  (K-\eps, K+\eps)\big) 
\le \liminf_{N\to \infty}
\PP\big(I_N \in  (K-\eps, K+\eps)\big) \\
& =  \liminf_{N\to \infty}
\nu_N\big(  (K-\eps, K+\eps)\big)\\
& \le \sup_{N \in \N} \bigg\|\frac{d\nu_N}{d\lambda}\bigg\|_{L^\infty(\R)} 
\cdot \ld \big(  (K-\eps, K+\eps)\big)\\
& \les \eps
\end{align*}

\noi
for any $\eps > 0$.
Since the choice of $\eps > 0$ was arbitrary, 
we then conclude \eqref{PP2}.
\end{proof}

\section{Non-normalizability of 
the focusing Gibbs measure with the quartic interaction}
\label{SEC:foc2}

In this section, we present the proof of 
the non-normalizability of the log-correlated  Gibbs
measure 
with the focusing quartic interaction
(Theorem \ref{THM:1}).

\subsection{Variational formulation}
\label{SUBSEC:var}

In order to  prove \eqref{Gibbs4} and \eqref{nonnormalisability}, 
we use a variational formula for the partition function
as in \cite{TW, OOT}.
Let us first introduce some notations.
Fix  
 a probability space $(\O, \F, \mathbb P)$.
Let $W(t)$ be a cylindrical Brownian motion in $L^2(\T^d)$.
Namely, we have
\begin{align}
W(t) = \sum_{n \in \Z^d} B_n(t) e_n,
\label{P1}
\end{align}

\noi
where  
$\{B_n\}_{n \in \Z^d}$ is a sequence of mutually independent complex-valued\footnote{By convention, we normalize $B_n$ such that $\text{Var}(B_n(t)) = t$. In particular, $B_0$ is  a standard real-valued Brownian motion.} Brownian motions such that 
$\cj{B_n}= B_{-n}$, $n \in \Z^d$. 
Then, define a centered Gaussian process $Y(t)$
by 
\begin{align}
Y(t)
=  \jb{\nabla}^{-\frac d2}W(t).
\label{P2}
\end{align}

\noi
Note that 
we have $\Law(Y(1)) = \mu$, 
where $\mu$ is the log-correlated Gaussian measure in \eqref{gauss0}.
By setting  $Y_N = \pi_NY $, 
we have   $\Law(Y_N(1)) = (\pi_N)_*\mu$, 
i.e.~the pushforward of $\mu$ under $\pi_N$.
In particular, 
we have  $\E [Y_N^2(1)] = \s_N$,
where $\s_N$ is as in~\eqref{sigma1}.
Here, the expectation $\E$ is with respect to 
the underlying probability measure $\PP$.

Next, let $\Ha$ denote the space of drifts, 
which are progressively measurable\footnote{Namely, 
the map $(t, \o) \in [0, 1] \times \O\mapsto  \ta(t, \o) \in L^2(\T^d)$
is $\mathcal B_{[0, t]}\otimes \F_t$-measurable, 
where $\mathcal B_{[0, t]}$ denotes the Borel sets in $[0, t]$
and $\{\F_t\}_{0 \le t \le 1}$ denotes the  filtration induced by the process  $Y$.
} 
processes 
belonging to 
$L^2([0,1]; L^2(\T^d))$, $\PP$-almost surely. 
We now state the  Bou\'e-Dupuis variational formula \cite{BD, Ust};
in particular, see Theorem 7 in \cite{Ust}.

\begin{lemma}\label{LEM:var3}
Let $Y$ be as in \eqref{P2}.
Fix $N \in \N$.
Suppose that  $F:C^\infty(\T^d) \to \R$
is measurable such that $\E\big[|F(\pi_NY(1))|^p\big] < \infty$
and $\E\big[|e^{-F(\pi_NY(1))}|^q \big] < \infty$ for some $1 < p, q < \infty$ with $\frac 1p + \frac 1q = 1$.
Then, we have
\begin{align}
- \log \E\Big[e^{-F(\pi_N Y(1))}\Big]
= \inf_{\dr \in \mathbb H_a}
\E\bigg[ F(\pi_N Y(1) + \pi_N I(\dr)(1)) + \frac{1}{2} \int_0^1 \| \dr(t) \|_{L^2_x}^2 dt \bigg], 
\label{P3}
\end{align}

\noi
where  $I(\dr)$ is  defined by 
\begin{align}
 I(\dr)(t) = \int_0^t \jb{\nabla}^{-\frac d2} \dr(t') dt'
\label{P3a}
\end{align}

\noi
and the expectation $\E = \E_\PP$
is an 
expectation with respect to the underlying probability measure~$\PP$.

\end{lemma}

In the following,
we construct a drift $\dr$ depending on $Y$
and 
 the  Bou\'e-Dupuis variational formula 
(Lemma \ref{LEM:var3}) is suitable for this purpose
since 
an  expectation in  \eqref{P3} is taken
with respect to  the underlying probability measure $\PP$.
Compare this with the variational formula
in~\cite{GOTW}, 
where an expectation is taken 
with respect to  a shifted measure.

Before proceeding to the proof of Theorem \ref{THM:1}, 
we state a lemma on the  pathwise regularity bounds  of 
$Y(1)$ and $I(\dr)(1)$.

\begin{lemma}  \label{LEM:Dr}
	
\textup{(i)} 
Let $\eps > 0$. Then, given any finite $p \ge 1$, 
we have 
\begin{align}
\E 
\Big[  \|Y_N(1)\|_{W^{-\eps,\infty}}^p
+ \|:\!Y_N^2(1)\!:\|_{W^{-\eps,\infty}}^p
+ 
\big\| :\! Y_N^3(1) \!:  \big\|_{W^{-\eps,\infty}}^p
\Big]
\leq C_{\eps, p} <\infty,
\label{P4}
\end{align}

\noi
uniformly in $N \in \N$.
	
\smallskip
	
\noi
\textup{(ii)} For any $\dr \in \Ha$, we have
\begin{align}
\| \I(\dr)(1) \|_{H^{\frac d2}}^2 \leq \int_0^1 \| \dr(t) \|_{L^2}^2dt.
\label{CM}
\end{align}
\end{lemma}

Before proceeding to the proof of Lemma \ref{LEM:Dr}, 
recall the following orthogonality result
\cite[Lemma 1.1.1]{Nua};
let $f$ and $g$ be jointly Gaussian random variables with mean zero 
and variances $\s_f$
and $\s_g$.
Then, we have 
\begin{align}
\E\big[ H_k(f; \s_f) H_\l(g; \s_g)\big] = \dl_{k\l} k! \big\{\E[ f g] \big\}^k, 
\label{Wickx}
\end{align}

\noi
where
 $H_k (x,\s)$ denotes the Hermite polynomial of degree $k$ with variance parameter $\s$.

\begin{proof}

Part (i) 
is a direct consequence of pathwise regularities
of the log-correlated Gaussian process $Y$ 
(and its Wick powers) whose law at time $t = 1$ is given by $\mu$ in \eqref{gauss0}.
See, for example, 
\cite[Proposition 2.3]{OTh2}
and 
\cite[Proposition 2.1]{GKO}
for related results when $d = 2$.
For readers' convenience, 
we present  details.
Given $\eps > 0$
and finite $p \ge 1$, let $r \ge p$
such that $\eps r > 2d$.
Then, 
from the Sobolev embedding theorem and Minkowski's integral inequality, 
we have 
\begin{align}
\begin{split}
 \Big\|\|:\!Y_N^k(1)\!:\|_{W^{-\eps, \infty}}\Big\|_{L^p(\O)}
& \les   \Big\|\|:\!Y_N^k(1)\!:\|_{W^{-\frac \eps2, r}}\Big\|_{L^p(\O)}\\
&  \leq\Big\|\|\jb{\nb}^{-\frac\eps2}:\!Y_N^k(1, x)\!:\|_{L^p(\O)}\Big\|_{L^r_x}.
\end{split}
\label{XH1}
\end{align}

\noi
On the other hand, 
from \eqref{Wick1} and \eqref{Wickx} with \eqref{P1} and \eqref{P2}, 
we have 
\begin{align*}
\E\big[:\!Y_N^k(1, x)\!:\, 
:\!Y_N^k(1, y)\!:\big]
& = k! \big\{\E[Y_N(1, x) Y_N(1, y)]\big\}^{k}\\
& = k!\sum_{\substack{n_1, \dots, n_k \in \Z^d\\ |n_j|\le N}}
\prod_{j = 1}^k \frac{1}{\jb{n_j}^d} e_{n_1+ \cdots + n_k} (x - y).
\end{align*}

\noi
By applying the Bessel potentials $\jb{\nb}_x^{-\frac\eps2}$
and $\jb{\nb}_y^{-\frac\eps2}$ of order $-\frac \eps2$
and then setting $x = y$, we have 
\begin{align}
\E\big[|\jb{\nb}^{-\frac\eps2} :\!Y_N^k(1, x)\!:|^2\big]
& = k!\sum_{\substack{n_1, \dots, n_k \in \Z^d\\ |n_j|\le N}}
\prod_{j = 1}^k \frac{1}{\jb{n_j}^d\jb{n_1+ \cdots + n_k}^\eps}
\les 1, 
\label{XH2}
\end{align}

\noi
uniformly in $N \in \N$.
Then, \eqref{P4}
follows from \eqref{XH1},  Lemma~\ref{LEM:hyp}, and \eqref{XH2}.

As for Part (ii), 
the estimate \eqref{CM}
follows  from \eqref{P3a}, Minkowski's inequality, and Cauchy-Schwarz's inequality.
See  the proof of Lemma 4.7 in \cite{GOTW} .
\end{proof}

\subsection{Proof of Theorem \ref{THM:1}}
\label{SUBSEC:3.2}

In this subsection, we present the proof of  Theorem \ref{THM:1}.
Let us first discuss 
the divergence
\eqref{nonnormalisability}
for  any $K>0$. 
Given $K, L > 0$ and $N \in \N$, define $Z_{K,L,N}$ and $Z_{K,L}$ by 
\begin{equation*}
Z_{K,L,N} = \E_\mu\Big[\exp\big(\min{( R_N(u),L)} \big) 
\cdot \ind_{\{ |\int_{\T^d} \, : \, u_N^2 :\, dx | \le K\}} \Big]
\end{equation*}
and 
\begin{equation*}
Z_{K,L} = \E_\mu\Big[\exp\big(\min{( R(u),L)} \big) 
\cdot \ind_{\{ |\int_{\T^d} \, : \, u^2 :\, dx | \le K\}} \Big].
\end{equation*}

\noi
Then, by the monotone convergence theorem, we have 
$$ Z_K = \lim_{L\to \infty} Z_{K,L}. $$

\noi
Moreover, 
by the dominated convergence theorem together with the almost sure convergence\footnote{See, for example, \cite[Proposition~1.1]{OTh} together with the Borel-Cantelli lemma.} 
 of $R_N(u)$ (and 
$\int_{\T^d}  : \! u_N^2 \!: dx$) to 
$R(u)$ (and $\int_{\T^d}  : \! u^2 \!: dx$, respectively)
and 
Lemma \ref{LEM:boundarymeasure}
(which guarantees
almost sure convergence of 
$\ind_{\{ |\int_{\T^d} \, : \, u_N^2 :\, dx | \le K\}}$
to $\ind_{\{ |\int_{\T^d} \, : \, u^2 :\, dx | \le K\}}$), we obtain 
$$ Z_{K,L} = \lim_{N \to \infty} Z_{K,L,N}. $$
Therefore, \eqref{nonnormalisability} follows once we prove
the following divergence:
\begin{align}
\lim_{L \to \infty} \liminf_{N \to \infty} 
Z_{K, L, N} =  \infty,
\label{pax}
\end{align}

\noi
where $R_N(u)$ is as  in \eqref{Wick0} with $\ld > 0$ and $k = 4$.

Noting that 
\begin{align}
Z_{K, L, N} \ge \E_\mu\Big[\exp\Big(\min{( R_N(u),L)} 
\cdot \ind_{\{ |\int_{\T^d} \, : u_N^2 :\, dx | \le K\}}\Big)   \Big]
- 1,
\label{pax1}
\end{align}

\noi
the  divergence \eqref{pax} (and thus \eqref{Gibbs4}) follows once we prove
\begin{align}
\lim_{L \to \infty} \liminf_{N \to \infty}  \E_\mu\Big[\exp\Big(\min{(  R_N(u),L)} 
\cdot \ind_{\{ |\int_{\T^d} \, : u_N^2 :\, dx | \le K\}}\Big)   \Big] =  \infty.
\label{pa0}
\end{align}

By  the  Bou\'e-Dupuis variational formula 
(Lemma \ref{LEM:var3}), we have
\begin{align}
\begin{split}
- \log & \, {\E_\mu \Big[\exp\Big(\min{(  R_N(u),L)} \cdot \ind_{\{ |\int_{\T^d} \, : u_N^2 :\, dx | \le K\}}\Big)   \Big]} \\
&= \inf_{\dr \in \mathbb H_a} \E\bigg[ -\min\big(  R_N(Y(1) + I(\dr)(1)),L\big)\\
&\hphantom{XXXXX} \times  \ind_{\{ |\int_{\T^d} : (\pi_N Y(1))^2: 
+ 2 (\pi _N Y(1)) (\pi_N I(\dr)(1)) + (\pi_N I(\dr)(1))^2 dx | \le K\}} \\
&\hphantom{XXXXX}
+ \frac 12 \int_0^1 \| \dr(t)\|_{L^2_x} ^2 dt \bigg],
\label{DPf}
\end{split}
\end{align}

\noi
where $Y(1)$ is as in \eqref{P2}.
Here,  $\E_\mu$ and $\E$ denote expectations
with respect to the Gaussian field~$\mu$ in \eqref{gauss0}
and the underlying probability measure $\PP$, respectively.
In the following, we show that the right-hand side 
of \eqref{DPf} tends to $-\infty$ as $N, L \to \infty$.
The main idea is to construct a drift $\dr$
such that 
$I(\dr)$ looks like ``$- Y(1) + $ 
a perturbation'', where the perturbation term is bounded in $L^2(\T^d)$
but has a large $L^4$-norm.\footnote{While we do not make use of solitons 
in an explicit manner in this paper, 
one should think of this perturbation as something like a soliton
or a finite blowup solution (at a fixed time) with a highly concentrated
profile whose $L^4$-norm blows up while its  $L^2$-norm remains bounded.
See Lemma \ref{LEM:leo1}.}

\medskip

\noi
$\bullet$ {\bf Part 1:}
We first present several preliminary results.
The proofs of Lemmas \ref{LEM:leo1} and \ref{LEM:leo2} are presented in Subsection \ref{SUBSEC:LEM}.
\\
\indent
We first construct a perturbation term in the next lemma.
Fix a large parameter $M \gg 1$.
Let $f: \R^d \to \R$ be a real-valued Schwartz function
with $\|f\|_{L^2(\R^d, \frac{dx}{(2\pi)^d})} = 1$
such that 
its Fourier transform $\ft f$ is 
supported  on $\{\xi \in \R^d:  |\xi| \le 1 \}$ with $\ft f(0) = 0$.
Define a function $f_M$  on $\T^d$ by 
\begin{align}
f_M =  M^{-\frac d2} \sum_{\substack{n \in \Z^d \\  |n| \le M}} \ft f\Big( \frac nM \Big) e_n, 
\label{fMdef} 
\end{align}

\noi
where $\ft f = \F_{\R^d}(f)$ denotes the Fourier transform on $\R^d$ defined in \eqref{FF1}.
Then, a direct computation yields  the following lemma.

\begin{lemma}\label{LEM:leo1}
Let   $\al > 0$. Then,  we have
\begin{align}
\int_{\T^d} f_M^2 dx &= 1 + O( M^{-\al}), \label{fM0} \\
\int_{\T^d}  f_M^4  dx &= M^d \|f\|_{L^4(\R^d, \frac{dx}{(2\pi)^d})}^4
+ O(M^{-\al}) \sim M^d,  \label{fM1}\\
\begin{split}
\int_{\T^d} (\jb{\nabla}^{-\al} f_M)^2 dx 
&\le 
 C(f)  M^{-d-2 + \max(d+ 2 -2\al, 0)}\\
&  = \begin{cases}
M^{-2\al}, & \text{for }\al \le \frac d2 + 1,\\
M^{-d-2}, & \text{for }\al > \frac d2 + 1.
 \end{cases}
 \end{split}
 \label{fm2} 
\end{align}

\noi
for any $M \gg 1$ and some constant $C(f) > 0$.

\end{lemma}

See Lemma 5.13 in \cite{OOT} for an analogous result on the construction
of a perturbation term.
While Lemma \ref{LEM:leo1} 
follows from a similar consideration, we present some details
of the proof 
in Subsection \ref{SUBSEC:LEM}.

\medskip

In the next lemma, we construct
an approximation $\z_M$ to $Y$ in \eqref{P2}
by solving stochastic differential equations.
Note that, in \cite{OOT}, such an approximation
of $Y(1)$
was constructed essentially by (a suitable frequency truncation of)
$Y(\frac 12)$, 
which was sufficient to prove a divergence analogous
to \eqref{pa0} for {\it large} $K \gg1 $.
In order to prove the divergence \eqref{pa0} for {\it any} $K > 0$, 
we need to establish a more refined approximation argument.
For simplicity, we denote  $Y(1)$ and $\pi_N Y(1)$ by $Y$ and $Y_N$,
respectively,  in the following.


\begin{lemma} \label{LEM:leo2}
Given $ M\gg 1$,  define $\z_M$ by its Fourier coefficients
as follows.
For $|n| \leq M$, $\ft \z_M(n, t)$ is a solution of the following  differential equation\textup{:}
\begin{align}
\begin{cases}
d \ft \z_{M}(n, t) = \jb{n}^{-\frac d2} M^\frac d2 (\ft Y(n, t)- \ft \z_{M}(n, t)) dt \\
\ft \z_{M}|_{t = 0} =0, 
\end{cases}
\label{ZZZ}
\end{align}

\noi
 and  we set $\ft \z_{M}(n, t)  \equiv 0$ for $|n| > M$.
Then, $\z_M(t)$ is a centered Gaussian process in $L^2(\T^d)$, which is frequency localized on $\{|n| \le M \}$, 
satisfying 
\begin{align}
&\E \big[ \z_M^2(x) \big] = \s_M(1 + o(1)) \sim \log M,\label{NRZ0}\\
&\E\bigg[  2 \int_{\T^d} Y_N \z_M dx - \int_{\T^d} \z_M^2 dx   \bigg] = \s_M(1 + o(1)) \sim \log M, \label{NRZ1}\\
&\E \bigg[  \Big|   \int_{\T^d} :\! ( Y_N-\z_M)^2 \!: dx  \Big|^2      \bigg] \les M^{-d}\log M,    \label{NRZ3}\\
&\E\bigg[\Big( \int_{\T^d} Y_N  f_M dx \Big)^2\bigg] 
+ \E\bigg[\Big( \int_{\T^d} \z_M  f_M dx \Big)^2\bigg] \les M^{-d},   \label{NRZ5}\\
&\E\bigg[\int_0^1 \Big\| \frac d {ds} \z_M(s) \Big\|^2_{H^\frac d 2}ds\bigg] \les M^d \label{NRZ6}
\end{align}
	
\noi
for any $N \ge M \gg 1$, 
where  $\z_M =\z_M|_{t = 1}$
and
\begin{align}
 :\! ( Y_N-\z_M)^2 \!: \, = 
 ( Y_N-\z_M)^2 - \E\big[ ( Y_N-\z_M)^2 \big].
\label{ZZZ2}
\end{align}

\noi
Here, \eqref{NRZ0} is independent of $x \in \T^d$. 
	
\end{lemma}

We now define $ \al_{M, N}$ by
\begin{align} 
 \al_{M, N}= \frac {\E \bigg[ 2 \int_{\T^d}Y_N\z_M dx-\int_{\T^d}\z_M^2  dx \bigg]}{\int_{\T^d} f_M^2 dx}.
\label{fmb1}
\end{align}

\noi
for $N\ge M \gg 1$.
Then, from \eqref{fM0} and \eqref{NRZ1}, we have
\begin{align}
 \al_{M, N} = \s_M(1+o(1)) \sim \log M
\label{logM}
\end{align}

\noi
for any $N \ge M\gg 1$.

\medskip

\noi
$\bullet$ {\bf Part 2:}
In this part, we  prove the divergence \eqref{pa0}.
For $M \gg 1$,
we set $f_M$, $\z_M$, and $ \al_{M, N}$ as in \eqref{fMdef}, Lemma \ref{LEM:leo2}, and \eqref{fmb1}.
For the minimization problem \eqref{DPf},
we set  a drift $\dr= \dr^0$ by 
\begin{align}
 \dr^0 (t) 
 & = \jb{\nb}^{\frac d2} \bigg( -\frac{d}{dt} \z_M(t) + \sqrt{ \al_{M, N}} f_M \bigg)
\label{paax}
\end{align}
	
\noi
such that 
\begin{align}
\Dr^0 = I(\dr^0)(1) 
= \int_0^1 \jb{\nb}^{-\frac d2} \dr^0(t) \, dt = - \z_M + \sqrt{ \al_{M, N}} f_M.
\label{paa0}
\end{align}

\noi
We also  define $Q(u$) by 
\begin{align}
Q(u) = \frac 14 \int_{\T^d}  u^4 dx
\qquad \text{and} \qquad Q_{\R^d}(v) = \frac 1{4(2\pi)^d} \int_{\R^d}  v^4 dx,
\label{paa1}
\end{align}
for $u \in L^4(\T^d)$ and  $v\in L^4(\R^d)$, respectively.

	
Let us first make some preliminary computations.
By Cauchy's inequality, we have
\begin{align}
\begin{split}
|\z_M(\sqrt{ \al_{M, N}} f_M)^3| &\le \frac \dl 4 \al_{M, N}^2 f_M^4 +\frac { 1}{\dl} \al_{M, N} \z_M^2 f_M^2, \\
|\z_M^3  \sqrt{ \al_{M, N}} f_M | &\le \frac \dl4 \z_M^4+\frac{ 1}{\dl} \al_{M, N}\z_M^2  f_M^2
\end{split}
\label{YO}
\end{align}

\noi
for any $0<\dl<1$.
Then, from 
\eqref{paa0}, \eqref{paa1}, and \eqref{YO}, we have 
\begin{align}
\begin{split}
 Q(\Dr^0) & -  \al_{M, N}^2 Q(f_M) \\
&= - \int_{\T^d} \z_M(\sqrt{ \al_{M, N}} f_M)^3  dx
+ \frac32 \int_{\T^d} \z_M^2 (\sqrt{ \al_{M, N}} f_M)^2 dx \\
&\quad 
- \int_{\T^d}  \z_M^3 \sqrt{ \al_{M, N}} f_M dx + Q(\z_M) \\
&\ge -\delta  \al_{M, N}^2 Q(f_M)
- C_\delta  \al_{M, N}\int_{\T^d}   \z_M^2 f_M^2  dx +(1-\delta) Q(\z_M)  \\
&\ge -\delta  \al_{M, N}^2 Q(f_M) - C_\delta  \al_{M, N} \int_{\T^d} 
 \z_M^2f_M^2  dx
\end{split}
\label{pa1}
\end{align}

\noi
for any $0<\dl<1$.
From \eqref{logM},  \eqref{NRZ0} in 
Lemma \ref{LEM:leo2}, 
and \eqref{fM0} in 
Lemma \ref{LEM:leo1}, 
we have
\begin{align}
\begin{split}
\E \bigg[ \al_{M, N} \int_{\T^d} \z_M^2 f_M^2  dx \bigg]
& = \al_{M, N} \int_{\T^d} \E [ \z_M^2(x)] f_M^2(x)  dx \\
&\sim  (\log M)^2 \| f_M \|_{L^2}^2
\les (\log M )^2
\end{split}
\label{pa2}
\end{align}

\noi
for any $N \ge M \gg 1$.
Therefore, 
it  follows from 
\eqref{pa1}, \eqref{pa2}, and \eqref{logM} with \eqref{paa1} and  \eqref{fM1}
that 
for any  measurable set  $E$ with $\PP(E) >0$
and  any  $L \gg \ld \cdot  \al_{M, N}^2 Q(f_M)$, we have
\begin{align}
\begin{split}
\E \Big[\min\big( \g \ld Q(\Dr^0), L\big) \cdot \ind_E\Big]
&\ge \g  \ld (1-\delta)  \al_{M, N}^2 Q(f_M) \PP(E) -  \g  C'_\delta (\log M )^2 \\
&=  \g \ld (1-\dl) \s_M^2 M^d Q_{\R^d}(f)\PP(E) (1 + o(1))  
\end{split}
\label{paa2}
\end{align}

\noi
for any $N \ge M \gg 1$.
	
Recall from \eqref{fMdef} that $\ft f_M$ is supported on $\{|n|\leq M\}$.
Then, 
by 
Lemma \ref{LEM:Dr}\,(ii) with 
\eqref{paa0}, \eqref{paax}, 
\eqref{NRZ6} in Lemma \ref{LEM:leo2}, 
\eqref{logM}, 
and \eqref{fM0} in 
Lemma \ref{LEM:leo1},  we have 
\begin{align}
\begin{split}
\E \big[ \| \Dr^0\|_{H^\frac d2}^2 \big]
&\le \E \bigg[ \int_0^1 \|\dr^0(t)\|_{L^2}^2  dt \bigg] \\
&\les  \E\bigg[\int_0^1 \Big\| \frac d {ds} \z_M(s) \Big\|^2_{H^\frac d 2}ds\bigg]
 + M^d  \alpha_{M,N} \|f_M \|_{L^2}^2  \\
&\les  M^d \log M. 
\end{split}
\label{pa4}
\end{align}

Lastly,
recall 
the following  identity (see \cite[(1.18)]{OTh2}):
\begin{align}
H_k(x+y; \s )
&  = 
\sum_{\l = 0}^k
\begin{pmatrix}
k \\ \l
\end{pmatrix}
 x^{k - \l} H_\l(y; \s), 
\label{Herm}
\end{align}

\noi
which follows from a Taylor expansion
with the differentiation rule
\cite[p.\,159]{Kuo}:
 $H_k(x;\s) = k H_{k-1}(x;\s)$.
Then, 
 from \eqref{Wick0} with $k = 4$
and \eqref{Herm}, 
we have 
\begin{align}
\begin{split}
R_N (Y + \Dr^0)  & = 
\frac \ld4\int_{\T^d}  :\! Y_N^4 \!:  dx
+ \ld \int_{\T^d}  :\! Y_N^3 \!:  \Dr^0 dx+\frac {3\ld}2\int_{\T^d}  :\! Y_N^2 \!:  (\Dr^0)^2 dx
\\
&\hphantom{X}
+ \ld\int_{\T^d} Y_N  (\Dr^0)^3 dx
+ \frac \ld4\int_{\T^d} (\Dr^0)^4 dx, 
\end{split}
\label{Y00}
\end{align}

\noi	
where we used  
\begin{align}
\pi_N \Dr^0 = \Dr^0
\label{YY0a}
\end{align}
for $N \ge M \ge 1$.
We now state a lemma, controlling the second, third, and fourth
terms on the right-hand side of~\eqref{Y00}.
We present the proof of this  lemma in Subsection \ref{SUBSEC:LEM}.

\begin{lemma} \label{LEM:Dr1}
There exist   small $\eps>0$ and   a constant  $c_0=c_0(\eps)  >0$ 
such that for any $\dl>0$, we have
\begin{align}
\begin{split}
\bigg| \int_{\T^d}  :\! Y_N^3 \!:  \Dr^0dx  \bigg|
&\le c(\dl) \| :\! Y_N^3 \!: \|_{W^{-\eps,\infty}}^2  
+ \dl
\| \Dr^0\|_{ H^{\frac d2}}^2,   
\end{split}
\label{Y1}\\
\bigg| \int_{\T^d}  :\! Y_N^2 \!: (\Dr^0)^2 dx \bigg|
&\le c(\dl) \| :\! Y_N^2 \!: \|_{W^{-\eps,\infty}}^4  
+ \dl \Big(
\| \Dr^0\|_{ H^{\frac d2}}^2 + 
\| \Dr^0 \|_{L^4}^4   \Big),   
\label{Y2} \\
\begin{split}
\bigg| \int_{\T^d}  Y_N (\Dr^0)^3    dx\bigg|
&\le c(\dl)  \| Y_N \|_{W^{-\eps,\infty}}^{c_0}+ \dl \Big(
\| \Dr^0\|_{ H^{\frac d2}}^2 + 
\| \Dr^0 \|_{L^4}^4   \Big), 
\end{split}
\label{Y3}
\end{align}

\noi
uniformly in $N \in \N$.

\end{lemma}

Fix small  $\dl_0 > 0$.
Then,  from \eqref{Y00} and Lemma \ref{LEM:Dr1}, 
we have 
\begin{align}
\begin{split}
 R_N(Y + \Dr^0)
&\ge (1-\dl_0) \ld Q(\Dr^0) \\
& \quad - c(\dl_0)\ld \Big(
\|  :\! Y_N^3 \!:  \|_{W^{-\eps,\infty}}^2
 + \|:\! Y_N^2 \!:\|_{W^{-\eps,\infty}}^4 
+  \|Y_N\|_{W^{-\eps,\infty}}^{c_0}
\Big) \\
&\quad
-c\dl_0 \ld \|\Dr^0\|_{H^\frac d2}^2-   |R_N(Y)|.
\end{split}
\label{paa}
\end{align}

We are now ready to put everything together.
With \eqref{YY0a} in mind, 
suppose that for any $K>0$ and small $\dl_1 >0$, there exists $M_0=M_0(K,\dl_1) \geq 1$ such that 
\begin{align}
\PP\bigg( \Big|\int_{\T^d} (:{Y_N^2}: + 2 Y_N \Dr^0 + (\Dr^0)^2) dx \Big| \le K \bigg) 
\ge 1 - \dl_1, 
\label{pa5}
\end{align}
	
\noi
uniformly in $N \ge M \ge M_0$.
Then,
it follows from 
\eqref{DPf},  \eqref{paa}, \eqref{paa2},  
Lemma \ref{LEM:Dr},  \eqref{pa4}, \eqref{conv1} (controlling $|R_N(Y)|$, uniformly in $N \in \N$), 
and \eqref{logM} 
with \eqref{YY0a}
that 
there exist constants $C_1, C_2 > 0 $ such that 
\begin{align}
 -\log & \, \E_\mu\Big[\exp\Big(\min{(  R_N(u),L)} \cdot \ind_{\{ |\int_{\T^d} \, : u_N^2 :\, dx | \le K\}}\Big)   \Big] \notag \\
&\le \E\bigg[ -\min\big( R_N(Y + \Dr^0),L\big)\notag \\
& \hphantom{XXX}\times
\ind_{\{ |\int_{\T^d} ( :{Y_N^2}: + 2 Y_N \Dr^0 + (\Dr^0)^2) dx | \le K\}} + \frac 12 \int_0^1 \| \dr^0(t)\|_{L^2_x} ^2 dt \bigg] \notag \\
&\le  \E\bigg[ -\min\big((1-\dl_0) \ld  Q(\Dr^0),L\big) \cdot \ind_{\{ |\int_{\T^d} ( :{Y_N^2}: + 2 Y_N \Dr^0 + (\Dr^0)^2) dx | \le K\}}  \notag \\
&\hphantom{XXX}
+ c(\dl_0)\ld \Big( 
\|  :\! Y_N^3 \!:  \|_{W^{-\eps,\infty}}^2
 + \|:\! Y_N^2 \!:\|_{W^{-\eps,\infty}}^4 
+ \|Y_N\|_{W^{-\eps,\infty}}^{c_0}
\Big)   \notag \\
&\hphantom{XXX}
+ c\dl_0\ld\|\Dr^0\|_{H^\frac d2}^2
+  |R_N(Y)|
+ \frac 12 \int_0^1 \| \dr^0(t)\|_{L^2_x} ^2 dt \bigg]  \notag \\
&\le -(1-\dl_0)(1-\dl)(1-\dl_1)\ld \alpha_{M,N}^2M^d Q_{\R^d}(f)(1+o(1)) \notag \\
& \quad  + C_1(\dl_0,\ld) M^d \log M+ C_2(\dl_0,\ld) \notag \\
&= -(1-\dl_0)(1-\dl)(1-\dl_1)\ld \s_M^2M^dQ_{\R^d}(f)(1+o(1)).
\label{U2}
\end{align}

\noi
for any  $N \ge M \ge M_0(K,\dl_1)$
and $L \gg 
\ld \cdot  \al_{M, N}^2 Q(f_M)\sim \ld  M^d (\log M)^2$.
Therefore,  we obtain
\begin{align}
\begin{split}
\lim_{L \to \infty} & \liminf_{N \to \infty} \E_\mu\Big[\exp\Big(\min{(\ld R_N(u),L)} \Big) 
\cdot \ind_{\{ |\int_{\T^d} \, : u_N^2 :\, dx | \le K\}} \Big]  \\
\ge&~  \exp\Big(  (1-\dl_0)(1-\dl)(1-\dl_1)\ld \s_{M}^2M^dQ_{\R^d}(f) (1+o(1))\Big) \too  \infty, 
\end{split}
\label{U3}
\end{align}
	
\noi
as $M \to \infty$.
This proves \eqref{pa0} by assuming \eqref{pa5}.

\medskip

It remains to 
 prove \eqref{pa5} for any $K> 0$ and small $\dl_1>0$. From \eqref{paa0}, we have 
\begin{align}
\begin{split}
&\E \bigg[ \Big|\int_{\T^d} \Big(  :{Y_N^2}: + 2 Y_N \Dr^0 + (\Dr^0)^2 \Big) dx \Big|^2 \bigg]\\
&= \E \bigg[ \Big|\int_{\T^d} :{Y_N^2}: dx - 2 \int_{\T^d} Y_N \z_M dx+\int_{\T^d} \z_M^2 dx+   \al_{M, N} \int_{\T^d} f_M^2 dx \\
&\hphantom{X}
 + 2 \sqrt{ \al_{M, N}} \int_{\T^d} (Y_N-\z_M)f_M dx \Big|^2 \bigg].
\end{split}
\label{Pr1}
\end{align}
	
\noi
From \eqref{logM} and \eqref{NRZ5} in Lemma \ref{LEM:leo2}, we have
\begin{align}
\E \bigg[ \Big| \sqrt{ \al_{M, N}}  \int_{\T^d} (Y_N-\z_M)f_M dx   \Big|^2   \bigg] \les 
 M^{-d} \log M.
\label{Pr4}
\end{align}	

\noi
On other hand, 
from \eqref{fmb1} and \eqref{ZZZ2}, we have	
\begin{align}
\begin{split}
\int_{\T^d} & :\!Y_N^2\!: dx - 2 \int_{\T^d} Y_N \z_M dx+\int_{\T^d} \z_M^2 dx+   \al_{M, N} \int_{\T^d} f_M^2 dx \\
& =\int_{\T^d} (Y_N-\z_M)^2 -\E \big[(Y_N-\z_M )^2 \big] dx \\
& =\int_{\T^d} :\! ( Y_N-\z_M)^2 \!: dx.
\end{split}
\label{Pr2} 
\end{align}	
	
\noi
Hence, from \eqref{Pr1}, \eqref{Pr4}, and \eqref{Pr2} with \eqref{NRZ3}
in Lemma \ref{LEM:leo2}, we obtain
\begin{align*}
\E \bigg[ \Big|\int_{\T^d} \Big(  :{Y_N^2}: + 2 Y_N \Dr^0 + (\Dr^0)^2 \Big) dx \Big|^2 \bigg]\les 
M^{-d }\log M.
\end{align*}

\noi
Therefore, by Chebyshev's inequality, 
given any $K > 0$ and small $\dl_1 > 0$, there exists $M_0 = M_0(K, \dl_1) \geq 1$ such that 
\begin{align*}
\PP\bigg( \Big|\int_{\T^d} (:{Y_N^2}: + 2 Y_N \Dr^0 + (\Dr^0)^2) dx \Big| > K \bigg)
&\le C\frac{M^{-d} \log M}{K^2}
< \dl_1
\end{align*}

\noi
for any $M \ge M_0 (K,\dl_1)$.
This proves  \eqref{pa5}.

\medskip

\noi
$\bullet$ {\bf Part 3:}
In this last part, we establish the exact divergence rate \eqref{divrate} of $Z_{K, N}$. From \eqref{U3} with $M=N$, we already have 
\begin{align}
\log Z_{K, N} \ge (1-\dl_0)(1-\dl)(1-\dl_1)\ld \s_N^2N^dQ_{\R^d}(f)(1+o(1))
\label{Low0}
\end{align} 
as $N\to \infty$,
for any small $\dl,\dl_0,\dl_1 > 0$ and any  Schwartz function $f$
with 
$\|f\|_{L^2(\R^d, \frac{dx}{(2\pi)^d})} = 1$, 
 $\supp(\ft{f}) \subset \{ |\xi|\le 1\}$, and $\ft f(0) = 0$.
Since Schwartz functions with  $\supp(\ft{f}) \subset \{ |\xi|\le 1\}$ 
and $\ft f(0) = 0$
are dense in 
$L^2(\R^d) \cap \big\{f: \supp(\ft{f}) \subset \{ |\xi|\le 1\}\big\}$, there exists a sequence 
$\{f_n\}_{n \in \N}$ of Schwartz functions with  $\|f_n\|_{L^2(\R^d, \frac{dx}{(2\pi)^d})} = 1$ 
and $\supp(\ft f_n) \subset \{ |\xi|\le 1\}$ 
which are  almost optimizers 
for Bernstein's inequality~\eqref{bernstein} on $\R^d$, namely, we have 
$$ \lim_{n \to \infty} Q_{\R^d}(f_n) = \frac{C_B}{4}. $$
Therefore, by inserting $f_n$ in \eqref{Low0}
and taking  $n \to \infty$ and $\dl,\dl_0,\dl_1 \to 0$, we obtain 
\begin{align}
\liminf_{N\to\infty} \frac{\log Z_{K, N}}{\s_N^2 N^d} \ge \ld \frac{C_B}{4}.
\label{Low1}
\end{align}

Hence, it remains to prove the upper bound. 
In view of  \eqref{DPf}, we have
\begin{align}
\begin{split}
\log Z_{K, N}&\le \sup_{\dr \in \Ha}\E\bigg[ R_N(Y + \Dr)\\
&\hphantom{XXX}
\times
\ind_{\{ |\int_{\T^d} ( :{Y_N^2}: + 2 Y_N \Dr_N + \Dr_N^2) dx | \le K\}} - \frac 12 \int_0^1 \| \dr(t)\|_{L^2_x} ^2 dt \bigg]\\
& \le \sup_{\dr \in \L^2_{t,x}}  \E\bigg[ R_N(Y+\Dr) \\
&\hphantom{XXX}
\times
 \ind_{\{ |\int_{\T^d} ( :{Y_N^2}: + 2 Y_N \Dr_N + (\Dr_N)^2) dx | \le K\}} - \frac 12 \int_0^1 \| \dr(t)\|_{L^2_x} ^2 dt \bigg]\\
& \le  \sup_{\Dr \in \mathcal{H}_x^{\frac d2} }  \E\bigg[ R_N(Y+\Dr) \cdot
\ind_{\{ |\int_{\T^d} ( :{Y_N^2}: + 2 Y_N \Dr_N + (\Dr_N)^2) dx | \le K\}} - \frac 12 \| \Dr_N  \|^2_{H^{\frac d2}} \bigg],
\end{split}
\label{rate0}  
\end{align}
\noi
where 
$\Dr = I(\dr)(1)$ in the first two lines and $\Dr_N = \pi_N \Dr$.
Here, the space
$\L^2_{t,x}$   denotes the space of drifts, which are stochastic processes belonging to 
$ L^2([0,1]; L^2(\T^d))$ 
 $\PP$-almost surely (namely, they do not have be  adapted), 
 and the space  $\mathcal{H}_x^{\frac d2}$
 denotes the space of 
$H^{\frac d2}(\T^d)$-valued random variables.

 For any $\Dr \in \mathcal{H}_x^\frac d2 $, let $V = Y+ \Dr$.
 Then, with $V_N = \pi_N V$, 
 we have
\begin{align}
\Dr_N=-Y_N+V_N,
\label{DR}
\end{align}
\noi
and thus we see that 
\begin{align}
\int_{\T^d} ( :{Y_N^2}: + 2 Y_N \Dr_N + \Dr_N^2) dx  \le K 
\quad \text{is equivalent to}\quad 
 \int_{\T^d} V_N^2 dx \le K +\s_N
\label{cut0}
\end{align}
\noi
where $\s_N=\E\big[Y_N^2\big]$
 is as in \eqref{sigma1}.
Hence, from \eqref{rate0}, 
a change of variables $\Dr_N=-Y_N+V_N$, 
\eqref{cut0}, and
the almost optimal Bernstein inequality (Lemma \ref{LEM:Bernstein}), we have 
\begin{align}
\begin{split}
\log Z_{K, N} &\le  \sup_{\Dr \in \mathcal{H}_x^{\frac d2} }  
\E\bigg[ R_N(Y+\Dr) 
\cdot   \ind_{\{ |\int_{\T^d} ( :{Y_N^2}: + 2 Y_N \Dr_N + (\Dr_N)^2) dx | \le K\}}  \bigg]\\
& \le  \sup_{V_N \in \mathcal{H}_x^{\frac d2} }
  \E\bigg[  R_N(V) 
  \cdot  \ind_{\{  \int_{\T^d} V_N^2 dx \le K +\s_N   \}}  \bigg]  \\
& \le \sup_{V_N \in \mathcal{H}_x^{\frac d2} } 
 \E\bigg[ \ld \frac {C_{\text{B}} }{4} N^d (1+o(1)) \| V_N \|_{L^2}^4 \cdot 
 \ind_{\{  \int_{\T^d} V_N^2 dx \le K +\s_N   \}}  \bigg] + O(\ld \s_N^2) \\
& \le \ld \frac{C_{B}}{4} N^d 
(1+o(1))
(K+\s_N)^2 + O(\ld \s_N^2) 
= \ld \frac{C_{B}}{4} N^d \s_N^2(1+o(1)) 
\end{split}
\label{rate2}
\end{align}
\noi
as $N\to \infty$, 
where, in the third step,  we used 
\[R_N(V) = 
\int_{\T^d}
\Big(\frac{\ld}{4}V_N^4 - \frac{3\ld}{2} \s_N V_N^2 + \frac{3\ld}{4} \s_N^2 \Big) dx 
\le 
\frac{\ld}{4} \int_{\T^d}V_N^4 dx  + \frac{3\ld}{4}\s_N^2.\]

\noi
 Therefore, combining this with \eqref{Low1}, we conclude \eqref{divrate}.

\begin{remark}\rm
The perturbation (at the level of $\Dr^0$ in \eqref{paa0})
is  given by $f_M$ (modulo the logarithmic factor $\sqrt {\al_{M, N}}$).
We point out that Lemma \ref{LEM:leo1}
shows that 
 $f_M$  looks like a highly concentrated profile whose $L^4$-norm
 (in fact, any $L^p$-norm for $p> 2$) blows up 
 while its $L^2$-norm is $O(1)$ as $M \to \infty$.
Note that  the blowup of $L^4$-norm
\eqref{fM1} 
   was crucially used in \eqref{paa2}, 
which led to the 
 desired divergence rate $M^d (\log M)^2$ in \eqref{U3}. 
 Moreover, the uniform (in $M$) bound~\eqref{fM0} 
 on the $L^2$-norm $f_M$
 played an essential role
 in \eqref{pa2} and \eqref{pa4} to guarantee that 
 the terms  in~\eqref{pa2} and \eqref{pa4} 
 grow at a slower rate than  $M^d (\log M)^2$.
\end{remark}

\subsection{Proofs of the auxiliary lemmas}
\label{SUBSEC:LEM}

In this subsection, we present the proofs of 
Lemmas~\ref{LEM:leo1}, \ref{LEM:leo2},  and~\ref{LEM:Dr1}.

We first briefly discuss the proof of Lemma \ref{LEM:leo1}.

\begin{proof}[Proof of Lemma \ref{LEM:leo1}]

Define a function $F_M$ on $\R^d$ by setting
\[ F_M(x) = M^\frac d2 f(Mx).\]

\noi
Then, 
from the Poisson summation formula \eqref{Poss} with \eqref{fMdef}, we
have 
\begin{align}
f_M(x) = \sum_{m\in \Z^d} F_M(x+2\pi m)
= \sum_{m\in \Z^d} T_mf(x), 
\label{poi1}
\end{align}

\noi
where $T_m f(x) = M^\frac d2 f(M(x+2\pi m))$.

Recall our convention of the normalized Lebesgue measure on $\T^d$.
Since $f$ is a Schwartz function, 
we have 
\begin{align}
\begin{split}
\int_{\T^d} (T_0f(x))^k dx 
& = \frac {M^{d( \frac k2 - 1)}}{(2\pi)^d}\int_{\R^d}
\ind_{[-\pi M,   \pi M)^d}(x)f^k(x)
 dx\\
&  = 
 \frac {M^{d( \frac k2 - 1)}}{(2\pi)^d}\int_{\R^d}f^k(x) dx
 + O(M^{-\al})
 \end{split}
 \label{poi2}
\end{align}

\noi
for any $\al > 0$.
On the other hand, from \eqref{poi1}, for $k \in \N$, we have 
\begin{align}
\begin{split}
\int_{\T^d} f_M^k(x) dx 
& = \int_{\T^d} \bigg(\sum_{m \in \Z^d} T_mf(x)\bigg)^k dx\\
& = 
\int_{\T^d} (T_0f)^k(x) dx + \text{l.o.t.}.
 \end{split}
 \label{poi3}
\end{align}

\noi
Here,  l.o.t. consists of the sum of the terms of the form 
\[\int_{\T^d} \prod_{j = 1}^k T_{m_j} f(x)  dx,\]

\noi
where $m_j \ne 0$ for at least one $j$.
It follows from  the fast decay of the Schwartz function $f$
that, for any $\kk > 0$, there exists $C > 0$ such that 
\begin{align*}
|T_m f(x)| = M^\frac d2 |f(M(x+2\pi m))|
\le C (Mm)^{-\kk}
\end{align*}

\noi
for any $m \in \Z^d\setminus\{0\}$;
see the proof of Lemma 5.13 in \cite{OOT}.
As a consequence, by summing over $m_j \in \Z^d$, $j = 1, \dots, k$
(not all zero), we obtain
\begin{align}
\begin{split}
|\,\text{l.o.t.}| \les M^{-\al}.
 \end{split}
 \label{poi4}
\end{align}

\noi
Therefore, 
from \eqref{poi2}, \eqref{poi3}, and \eqref{poi4}
with $\|f\|_{L^2(\R^d, \frac{dx}{(2\pi)^d})} = 1$, 
we conclude 
\eqref{fM0} and 
\eqref{fM1}.

Next, we prove \eqref{fm2}.
Since $f$ is a Schwartz function with $\ft f (0) = 0$, 
it follows from the fundamental theorem of calculus that 
\begin{align}
|\ft f(\xi)|
= |\ft f(\xi) - \ft f(0)|
 \le C_f |\xi|
\label{poi5}
\end{align}

\noi
for any $\xi \in \R^d$.
By Plancherel's identity with \eqref{fMdef} and \eqref{poi5}, we have 
\begin{align*}
\int_{\T^d} (\jb{\nabla}^{-\al} f_M)^2 dx 
& = M^{-d} \sum_{\substack{n \in \Z^d\\|n|\le M}}\Big|\ft f\Big(\frac nM\Big)\Big|^2 \frac 1{\jb{n}^{2\al}}\\
& \le C^2_f  M^{-d-2} \sum_{\substack{n \in \Z^d\\|n|\le M}} \frac 1{\jb{n}^{2(\al-1)}}\\
& \les C^2_f  M^{-d-2 + \max(d+ 2 -2\al, 0)}.
\end{align*}

\noi
This prove \eqref{fm2}.
\end{proof}

Next, we present the proof of the approximation lemma (Lemma~\ref{LEM:leo2}).

\begin{proof}[Proof of Lemma \ref{LEM:leo2}]
Let 
\begin{align}
X_n(t)=\ft Y_N(n, t)- \ft \z_{M}(n, t), 
\quad |n|\le M.
\label{ZZ1} 
\end{align}
Then,  from \eqref{P2} and \eqref{ZZZ}, 
we see that $X_n(t)$ satisfies 
the following stochastic differential equation:
\begin{align*}
\begin{cases}
dX_n(t)=-\jb{n}^{-\frac d 2}M^\frac d2 X_n(t) dt +\frac{1}{\jb{n}^\frac d2}dB_n(t)\\
X_n(0)=0
\end{cases}
\end{align*}	

\noi
for $|n|\le M$.
By solving this stochastic differential equation, we have
\begin{align}
X_n(t)=\frac{1}{\jb{n}^\frac d2}\int_0^t e^{-\jb{n}^{-\frac d 2}M^\frac d2(t-s)}dB_n(s).
\label{ZZ2}
\end{align}

\noi
Then, from \eqref{ZZ1} and \eqref{ZZ2}, we have 
\begin{align}
\ft \z_{M}(n, t)= \ft Y_N(n, t)-\frac{1}{\jb{n}^\frac d2 }\int_0^t e^{-\jb{n}^{-\frac d 2}M^\frac d2 (t-s)}dB_n(s)
\label{SDE1}
\end{align}

\noi
for $|n|\le M$. Hence, from \eqref{SDE1}, the independence of $\{B_n \}_{n \in \Z^d}$,\footnote{Here, we are referring to the independence modulo the condition $\cj{B_n} = B_{-n}$, $n \in \Z^d$.
Similar comments apply in the following.} 
Ito's isometry, and \eqref{P2}, we have
\begin{align}
\begin{split}
\E \big[ |\z_M(x)|^2 \big]&=\sum_{|n| \le M} \bigg( \E \big[  | \ft Y_N(n) |^2  \big]
-\frac 2{\jb{n}^d}\int_0^1 e^{-\jb{n}^{-\frac d 2}M^\frac d2 (1-s)}ds \\
& \hphantom{XXXXX}
+\frac{1}{\jb{n}^d}\int_0^1 e^{-2\jb{n}^{-\frac d 2}M^\frac d2 (1-s)}ds \bigg)\\
& = \s_M + O\Big(\sum_{|n|\le M }\frac 1{\jb{n}^\frac d2 }\cdot  \frac{1}{M^\frac d2}\Big)\\
& = \s_M(1 + o(1)).
\end{split}
\label{ZZ3}
\end{align}

\noi
for any $M\gg 1$.
This proves \eqref{NRZ0}.

By Parseval's theorem, \eqref{SDE1}, \eqref{NRZ0}, and  
proceeding as in \eqref{ZZ3}, we have
\begin{align*}
\E\bigg[  2 \int_{\T^d} Y_N & \z_M dx - \int_{\T^d} \z_M^2 dx   \bigg]
=\E \bigg[ 2 \sum_{|n| \le M  }\ft Y_N(n)  \cj{ \ft \z_M(n)} -\sum_{|n |\le M }|  \ft \z_M(n) |^2    \bigg]\\
&=\E \bigg[ \sum_{|n | \le M } | \ft \z_M(n)   |^2+\sum_{|n|\leq M } \bigg( \frac 2{\jb{n}^\frac d2 } \int_0^1 e^{-\jb{n}^{-\frac d 2}M^\frac d2 (1-s) }dB_n(s) \bigg)  \cj{\ft \z_M(n)}   \bigg]\\
&= \s_M(1 + o(1)) + O\Big(\sum_{|n| \le M } \frac 1{\jb{n}^\frac d2}\cdot \frac 1{M^\frac d2}\Big)\\
& = \s_M(1 + o(1)).
\end{align*} 

\noi 
for any $N\ge M\gg 1$.
This proves  \eqref{NRZ1}.

Note that  $\ft Y(n)-\ft \z_M(n)$ is a mean-zero Gaussian random variable.
Then, from \eqref{SDE1}
and Ito's isometry, 
we have 
\begin{align}
\begin{split}
\E \bigg[ \Big( |\ft Y_N(n)- & \ft \z_M(n)|^2-\E\big[ |\ft Y(n)-\ft \z_M(n)|^2 \big] \Big)^2  \bigg]
 \les 
\Big(\E\big[ |\ft Y_N(n)-\ft \z_M(n)|^2 \big] \Big)^2  \\
& = \frac 1 {\jb{n}^{2d} } \bigg(\int_0^1 e^{-2\jb{n}^{-\frac d 2}M^{\frac d2}(1-s)  } ds \bigg)^2 
\sim \frac 1 {\jb{n}^{d} } \cdot \frac 1 {M^d}.
\end{split}
\label{ZZ4}
\end{align}

\noi
Hence, 
from Plancherel's identity, \eqref{ZZZ2}, 
 the independence of $\{B_n \}_{n \in \Z^d}$, 
 the independence
of 
$\big\{ |\ft Y_N(n) |^2- \E \big[ | \ft Y_N(n)|^2 \big]\big\}_{M < |n|\le N}$
and 
\[\big\{ |\ft Y_N(n)-\ft \z_M(n)|^2-\E\big[ |\ft Y_N(n)-\ft \z_M(n)|^2\big]\big\}_{|n|\le M},\]

\noi 
 \eqref{P2}, 
 and \eqref{ZZ4}, we have
\begin{align*}
\E \bigg[  & \Big|   \int_{\T^d} :\! ( Y_N-\z_M)^2 \!: dx  \Big|^2      \bigg] \notag \\
&= \sum_{M<|n|\le N } \E \bigg[ \Big( |\ft Y_N(n) |^2- \E \big[ | \ft Y_N(n)|^2 \big] \Big)^2 \bigg]\notag \\
&\hphantom{XX}+ \sum_{|n|\le M} \E \bigg[ \Big( |\ft Y_N(n)-\ft \z_M(n)|^2-\E\big[ |\ft Y_N(n)-\ft \z_M(n)|^2 \big] \Big)^2  \bigg]\notag  \\
&\les \sum_{M<|n|\le N}\frac 1{\jb{n}^{2d}}
+\sum_{|n|\le M }
\frac 1{\jb{n}^{d}}
\frac1 {M^d}
\les M^{-d}\log M.
\end{align*}

\noi
This proves \eqref{NRZ3}.

From \eqref{fm2} and \eqref{P2}, we have
\begin{align}
\begin{split}
\E\bigg[ \Big( \int_{\T^d} Y_N  f_M dx  \Big)^2\bigg]
&= \E \bigg[ \Big( \sum_{|n| \le M}  \ft Y_N(n) \cj{ \ft f_M(n)} \Big)^2 \bigg]
= \sum_{|n| \le M} \frac 1{\jb{n}^d} |\ft f_M(n)|^2 \\
&\le \int_{\T^d} \big(\jb{\nb}^{-\frac d2} f_M (x)\big)^2 dx
\les M^{-d}.
\end{split}
\label{app4}
\end{align}

\noi
From \eqref{ZZ2}, Ito's isometry, and \eqref{fm2}, we have 
\begin{align}
\begin{split}
 \E \bigg[ \Big( \sum_{|n| \le M} X_n(1) \cj{ \ft f_M(n)} \Big)^2 \bigg]
&=\E \Bigg[ \bigg|\sum_{|n|\leq M} \bigg( \frac 1{\jb{n}^\frac d2}  \int_0^1 e^{-\jb{n}^{-\frac d 2} M^{\frac d2}(1-s) } dB_n(s) \bigg)  \ft f_M(n)     \bigg|^2     \Bigg]\\
&\les M^{-\frac d2}\sum_{|n| \leq M} \frac 1{\jb{n}^\frac d2 }| \ft f_M(n)|^2\\
& \les M^{-d}.
\end{split}
\label{app5}
\end{align}

\noi
Hence, \eqref{NRZ5} follows
from \eqref{app4} and \eqref{app5}
with \eqref{SDE1}.

Lastly, from \eqref{ZZZ}, \eqref{ZZ1}, \eqref{ZZ2}, and Ito's isometry, we have 
\begin{align*}
\E\bigg[\int_0^1 \Big\| \frac d {ds} \z_M(s) \Big\|^2_{H^\frac d 2}ds\bigg] &= M^d \E\bigg[\int_0^1 \Big\| \pi_M(Y_N(s)) - \z_M(s) \Big\|^2_{L^2}ds\bigg] \\
&= M^d \E\bigg[ \int_0^1 \Big(\sum_{|n| \le M} |X_n(s)|^2\Big)  ds\bigg]\\
&=M^d \sum_{|n| \le M} \frac 1 {\jb{n}^{d}} \int_0^1 \int_0^s e^{-2\jb{n}^{-\frac d 2}M^\frac d2(s-s')} d s' ds \\
& \les M^d \sum_{|n| \le M } \frac 1{\jb{n}^\frac d2}\cdot \frac 1{M^\frac d2}\\
&  \les M^d, 
\end{align*}

\noi
yielding  \eqref{NRZ6}.
This  completes the proof of Lemma~\ref{LEM:leo2}.
\end{proof}

Finally, we 
present the proof of Lemma \ref{LEM:Dr1}.

\begin{proof}[Proof of Lemma \ref{LEM:Dr1}]
	
From the duality and Cauchy's inequality,  
we have
\begin{align}
\begin{split}
\bigg| \int_{\T^d}  :\! Y_N^3 \!: \Dr^0 dx \bigg|
&\le  \| :\! Y_N^3 \!:\|_{W^{-\eps,\infty}} \| \Dr^0\|_{W^{\eps,1} }\le  \| :\! Y_N^3 \!:\|_{W^{-\eps,\infty}} 
\| \Dr^0\|_{H^\frac d2}\\
&\le c(\dl)\| :\! Y_N^3 \!:\|_{W^{-\eps,\infty}}^2 + \dl \| \Dr^0\|_{H^\frac d2}^2. 
\end{split}
\label{ZZZ3}
\end{align}
	
\noi
This yields  \eqref{Y1}.

From the fractional Leibniz rule (Lemma \ref{LEM:prod}\,(ii)), we have 
\begin{align}
\begin{split}
\bigg| \int_{\T^d}  :\! Y_N^2 \!: (\Dr^0)^2 dx \bigg|
&  \le  \| :\! Y_N^2 \!:\|_{W^{-\eps,\infty}}
\|(\Dr^0)^2\|_{W^{\eps, 1}}\\
& \leq
 \| :\! Y_N^2 \!:\|_{W^{-\eps,\infty}} 
\|(\Dr^0)^2\|_{W^{\eps, \frac{4}{3}}}\\
& \les  \| :\! Y_N^2 \!:\|_{W^{-\eps,\infty}}\|\Dr^0\|_{H^\frac d2}\|\Dr^0\|_{L^4}.
\end{split}
\label{ZZZ4}
\end{align}

\noi
Then, the second estimate \eqref{Y2}
follows from Young's inequality.

Lastly, we consider \eqref{Y3}.
From the fractional Leibniz rule 
(Lemma \ref{LEM:prod}\,(ii))
(with $\frac{1}{1+\dl} = \frac 1{2+\dl_0} + \frac 14 + \frac 14$
for small $\dl, \dl_0 > 0$), 
 Sobolev's inequality, 
 and the interpolation (Lemma \ref{LEM:prod}\,(i)),
we  have
 \begin{align}
 \begin{split}
\bigg| \int_{\T^d}    Y_N   (\Dr_N^0)^3 dx \bigg|
&\le \| Y_N \|_{W^{-\eps,\infty}} \| \jb{\nb}^{\eps}  (\Dr_N^0)^3\|_{L^{1+\dl}}\\
&\les\| Y_N \|_{W^{-\eps,\infty}} 
\| \Dr_N^0 \|_{W^{\eps, 2+\dl_0}} \| \Dr_N^0 \|_{L^4}^2\\
&\les\| Y_N \|_{W^{-\eps,\infty}} 
\| \Dr_N^0 \|_{H^\frac{d}{2}}^\be \| \Dr_N^0 \|_{L^4}^{3-\be}
\end{split}
\label{ZZZ2a}
\end{align}

\noi
for some small $\be > 0$.
Then, the third  estimate \eqref{Y3}
follows from Young's inequality
since $\frac{\be}2 + \frac {3-\be}4 < 1$ for small $\be > 0$.
This completes the proof of Lemma \ref{LEM:Dr1}.
\end{proof}

\section{Construction of the Gibbs measure with the cubic interaction}
\label{SEC:4}

In this section, we present the proof of Theorem \ref{THM:2}. 
We  prove  the uniform exponential integrability 
\eqref{exp3} 
via the variational formulation. 
Since the argument is identical for any finite $p \geq 1$, we only present details for the case $p =1$. 
Moreover,  the precise value of $\ld \in \R\setminus \{0\}$ does not play any role
and thus we set $\ld = 3$ in the following.

In view of the Bou\'e-Dupuis formula (Lemma \ref{LEM:var3}), 
it suffices to  establish a  lower bound on 
\begin{equation}
\W_N(\dr) = \E
\bigg[-R_N^\diamond(Y(1) + I(\dr)(1)) + \frac{1}{2} \int_0^1 \| \dr(t) \|_{L^2_x}^2 dt \bigg], 
\label{v_N0}
\end{equation}

\noi 
uniformly in $N \in \N$ and  $\dr \in \Ha$.
We  set $Y_N = \pi_N Y = \pi_N Y(1)$ and $\Dr_N = \pi_N  \Dr = \pi_N I(\dr)(1)$.

From \eqref{K2} and \eqref{Herm}, we have
\begin{align}
\begin{split}
R_N^\diamond (Y + \Dr)  & = 
\int_{\T^d}  :\! Y_N^3 \!:  dx
+ 3\int_{\T^d}  :\! Y_N^2 \!:  \Dr_N dx+3 \int_{\T^d}   Y_N   \Dr_N^2 dx
\\
&\hphantom{X}
+ \int_{\T^d} \Dr_N^3 dx
-  A \bigg\{ \int_{\T^d} \Big( :\! Y_N^2 \!: + 2 Y_N \Dr_N + \Dr_N^2 \Big) dx \bigg\}^2. 
\end{split}
\label{Y0}
\end{align}

\noi
Hence, from  \eqref{v_N0} and \eqref{Y0}, we have
\begin{align}
\begin{split}
\W_N(\dr)
&=\E
\bigg[
-\int_{\T^d}  :\! Y_N^3 \!:  dx
-3\int_{\T^d}  :\! Y_N^2 \!:  \Dr_N dx
-3\int_{\T^d}   Y_N   \Dr_N^2 dx 
\\
&\hphantom{XXX}
-\int_{\T^d} \Dr_N^3 dx
+  A \bigg\{ \int_{\T^d} \Big( :\! Y_N^2 \!: + 2 Y_N \Dr_N + \Dr_N^2 \Big) dx \bigg\}^2\\
&\hphantom{XXX}
+ \frac{1}{2} \int_0^1 \| \dr(t) \|_{L^2_x}^2 dt 
\bigg].
\end{split}
\label{v_N0a}
\end{align}

In the following, we first state 
a lemma, controlling the  terms appearing in \eqref{v_N0a}.
We present the proof of this lemma at the end of this section.

\begin{lemma} \label{LEM:Dr2}
\textup{(i)}
There exist small $\eps>0$ and  a constant  $c  >0$ such that
\begin{align}
\bigg| \int_{\T^d}  :\! Y_N^2 \!: \Dr_N dx \bigg|
&\le c \| :\! Y_N^2 \!: \|_{W^{-\eps,\infty}}^2  
+ \frac 1{100} 
\| \Dr_N \|_{H^\frac d2}^2, 
\label{YY2} \\
\bigg| \int_{\T^d}  Y_N \Dr_N^2  dx \bigg|
&\le c 
\| Y_N \|_{W^{-\eps,\infty}}^{6} + \frac 1{100} \Big(
\| \Dr_N \|_{H^\frac d2}^2 +  \| \Dr_N \|_{L^2}^4 \Big),
\label{YY3}\\
\bigg| \int_{\T^d}  \Dr_N^3  dx \bigg|
&\le     \frac 1{100} \| \Dr_N \|_{H^\frac d2}^2
+ \frac{A}{100} \| \Dr_N \|_{L^2}^{4 }
\label{YY14}
\end{align}

\noi
for any  sufficiently large $A>0$,  
uniformly in $N \in \N$.

\smallskip

\noi
\textup{(ii)}	
Let $A> 0$. Given any small $\eps > 0$, 
there exists $c = c(\eps, A)>0$ such that
\begin{align}
\begin{split}
A\bigg\{ \int_{\T^d}&  \Big( :\! Y_N^2 \!: + 2 Y_N \Dr_N + \Dr_N^2 \Big) dx \bigg\}^2 \\
&\ge \frac A4 \| \Dr_N \|_{L^2}^4 - \frac 1{100} \| \Dr_N \|_{H^\frac d2}^2 
- c\bigg\{ \| Y_N \|_{W^{-\eps,\infty}}^c 
+ \bigg( \int_{\T^d} :\! Y_N^2 \!:  dx \bigg)^2 \bigg\}, 
\end{split}
\label{YY5}
\end{align}

\noi
uniformly in $N \in \N$.

\end{lemma}

As in \cite{BG, GOTW, ORSW2, OOT}, 
the main strategy is 
to establish a pathwise lower bound on $\W_N(\dr)$ in~\eqref{v_N0a}, 
uniformly in $N \in \N$ and $\dr \in \Ha$, 
by making use of the 
 positive terms:
\begin{equation}
\U_N(\dr) =
\E \bigg[\frac A 4\| \Dr_N\|_{L^2}^4 + \frac{1}{2} \int_0^1 \| \dr(t) \|_{L^2_x}^2 dt\bigg]
\label{v_N1}
\end{equation}

\noi
coming from \eqref{v_N0a} and \eqref{YY5}.
From \eqref{v_N0a} and \eqref{v_N1} together with Lemmas  \ref{LEM:Dr2} and \ref{LEM:Dr}, we obtain
\begin{align}
\inf_{N \in \mathbb{N}} \inf_{\dr \in \Ha} \W_N(\dr) 
\geq 
\inf_{N \in \mathbb{N}} \inf_{\dr \in \Ha}
\Big\{ -C_0 + \frac{1}{10}\U_N(\dr)\Big\}
\geq - C_0 >-\infty.
\label{YYY5}
\end{align}

\noi
Then,  
 the uniform exponential integrability \eqref{exp3} 
follows from 
\eqref{YYY5} and Lemma \ref{LEM:var3}.
This proves 
 Theorem \ref{THM:2}.

\medskip

We conclude this section by 
presenting the proof of Lemma \ref{LEM:Dr2}.

\begin{proof}[Proof of Lemma \ref{LEM:Dr2}]
(i)
The estimate \eqref{YY2} follows
from replacing $:\! Y_N^3 \!:$ in \eqref{ZZZ3}
by $:\! Y_N^2 \!:$.

With small $\dl > 0$, 
it follows from 
the fractional Leibniz rule 
(Lemma \ref{LEM:prod}\,(ii))
and  Sobolev's inequality
as in \eqref{ZZZ2a} that 
\begin{align*}
\bigg| \int_{\T^d}   Y_N \Dr_N^2 dx \bigg|
&  \le  \|  Y_N\|_{W^{-\eps,\infty}}
\|\Dr_N^2\|_{W^{\eps, 1+\dl}}\\
& \leq
 \|  Y_N\|_{W^{-\eps,\infty}} 
\|\Dr_N\|^2_{H^\eps}\\
& \les  \|  Y_N\|_{W^{-\eps,\infty}}
\|\Dr_N\|_{H^\frac d2}^{\be} \|\Dr_N\|_{L^2}^{2-\be}
\end{align*}

\noi
for some small $\be> 0$.
Then, 
the second estimate \eqref{YY3} follows
from Young's inequality since
$\frac{\be} 2 + \frac{2-\be}4 < 1$.

As for the third estimate 
\eqref{YY14}, 
it follows from 
Sobolev's inequality, Lemma \ref{LEM:prod}\,(i),  and Cauchy's  inequality that 
\begin{align*}
\begin{split}
\bigg| \int_{\T^d}  \Dr_N^3  dx \bigg|
&\le C  \| \Dr_N \|_{H^\frac d6}^3 
\le C \| \Dr_N \|_{H^\frac d2}\| \Dr_N \|_{L^2}^2  \\
&\le \frac 1{100}\| \Dr_N \|_{H^\frac d2}^2 + \frac A{100} \| \Dr_N \|_{L^2}^4,
\end{split}
\end{align*}
	
\noi
where $A>0$ is sufficiently large.

\medskip

\noi
(ii) The bound \eqref{YY5} follows from a slight modification
of Lemma 5.8 in \cite{OOT}.
Noting that 
\begin{align*}
(a+b+c)^2
\ge \frac 12 c^2 - 2 (a^2+b^2)
\end{align*}
	
\noi
for any $a,b,c \in \R$, 
  we have 
\begin{align}
\begin{split}
A\bigg\{ & \int_{\T^d} \Big( :\! Y_N^2 \!: + 2 Y_N \Dr_N + \Dr_N^2 \Big) dx \bigg\}^2 \\
&\ge \frac A2 \bigg( \int_{\T^d} \Dr_N^2dx \bigg)^2
 - 2A \bigg\{ \bigg( \int_{\T^d} :\! Y_N^2 \!: dx \bigg)^2
+ \bigg( \int_{\T^d} Y_N \Dr_N dx \bigg)^2 \bigg\}.
\end{split}
\label{YZ3}
\end{align}

\noi
From Lemma \ref{LEM:prod}\,(i) 
 and Young's inequality, we have
\begin{align}
\begin{split}
\bigg| \int_{\T^d} Y_N \Dr_N dx \bigg|^2
&\le \| Y_N \|_{W^{-\eps,\infty}}^2 \| \Dr_N \|_{W^{\eps,1}}^2  \le \| Y_N \|_{W^{-\eps,\infty}}^2 \| \Dr_N \|_{H^\eps}^2\\
&\les \| Y_N \|_{W^{-\eps,\infty}}^2 \| \Dr_N \|_{L^2}^{2-\frac{4\eps}{d} } \| \Dr_N \|_{H^\frac d2}^{\frac {4\eps}d} \\
&\le c \| Y_N \|_{W^{-\eps,\infty}}^{\frac{2d}{d-2\eps}} + \frac 1{8} \| \Dr_N \|_{L^2}^4 + \frac1{200 A} \| \Dr_N \|_{H^\frac d2}^2.
\end{split}
\label{YZ4}
\end{align}
	
\noi
Hence, \eqref{YY5} follows from \eqref{YZ3} and \eqref{YZ4}.
\end{proof}

\begin{remark}\label{REM:exp}\rm

In considering the construction of the 
Gibbs measure with the cubic interaction, 
it is possible to consider
the following renormalized potential energy
with a general power $\g > 0$
on the Wick-ordered $L^2$-norm:
\begin{align}
R_N^{\diamond, \g} (u)
&=  \frac \ld 3\int_{\T^d}  :\! u_N^3 \!:  dx
- A \, \bigg( \int_{\T^d} :\! u_N^2 \!: dx\bigg)^\g, 
\label{ga0}
\end{align}

\noi
\noi
where the coupling constant
$\ld \in \R\setminus  \{0\} $
 denotes the strength of cubic  interaction
as in \eqref{K2}.
When $\g = 2$, 
$R_N^{\diamond, \g} (u)$ reduces to 
$R_N^\diamond (u)$ in \eqref{K2}.

In the following, 
let us briefly discuss the optimality of the power $\g = 2$ in Theorem \ref{THM:2}.
In view  of \eqref{YY3} and \eqref{YY14}, we need to control the term  $\| \Dr_N \|_{L^2}^4$, which forces us to choose $\g \ge 2$ in~\eqref{ga0}. 
When $\g = 2$, 
 it is also necessary to choose $A$ sufficiently large because of \eqref{YY3}.
 When $\g<2$ or when $\g=2$ and $A$ is sufficiently small, the taming by the Wick-ordered $L^2$-norm 
 in~\eqref{ga0} 
 is too weak to control the terms mentioned above,
 and 
thus we expect an analogous non-normalizability result to hold by repeating the proof of Theorem \ref{THM:1}.
\end{remark}

\appendix

\section{On the Gibbs measure for the two-dimensional Zakharov system}
\label{SEC:A}

In this appendix, we give a brief discussion on
Gibbs measures for the following 
scalar Zakharov system on $\T^d$:
\begin{align}
\begin{cases}
i \dt u +\Dl u = uw\\
c^{-2} \dt^2 w -  \Dl w = \Dl(|u|^2) 
\end{cases}
\label{Zak1}
\end{align}

\noi
This is a coupled system of Schr\"odinger and wave equations.
The unknown $u$ for the Schr\"odinger part is complex-valued, 
while the unknown $w$ for the wave part is real-valued.
By introducing 
the velocity field $\vec v$:
\begin{align*}
\dt w = - c^{2}\nb \cdot \vec v, 
\end{align*}

\noi
we can rewrite  \eqref{Zak1}
as 
\begin{align}
\begin{cases}
i \dt u +\Dl u = uw\\
\dt w = - c^{2}\nb \cdot \vec v\\
\dt \vec v = -\nb w - \nb (|u|^2).
\end{cases}
\label{Zak3}
\end{align}

\noi
Note that \eqref{Zak3} is a Hamiltonian system
with the Hamiltonian
\begin{align}
H(u, w, \vec v) = \frac{1}{2} \int_{\T^d}
\big( |\nb u|^2 + |u|^2 w \big) dx
+ \frac{1}{4} \int_{\T^d} w^2 dx 
+ \frac{c^2}{4} \int_{\T^d} |\vec v|^2 dx. 
\label{Zak4}
\end{align}

\noi
Moreover, 
the wave energy, namely, the $L^2$-norm of the Schr\"odinger component: 
\begin{align*}
M(u)  = \int_{\T^d} |u|^2 dx
\end{align*}

\noi
is known to be conserved.
See \cite{CCS}.

By setting  $W = \frac{1}{\sqrt2}w$
and $\vec V = (V_1, \dots, V_d) = \frac{c}{\sqrt2}\vec v$, 
 we can rewrite the Hamiltonian in \eqref{Zak4} as  
\begin{align}
H(u, W, \vec V) = \frac{1}{2} \int_{\T^d}
\big(|\nb u|^2 + \sqrt 2 |u|^2 W \big) dx 
+ \frac{1}{2} \int_{\T^d} W^2 dx 
+ \frac{1}{2} \int_{\T^d} |\vec V|^2 dx. 
\label{Zak5}
\end{align}

\noi
Then, the  Gibbs measure for the system \eqref{Zak3}
is formally given by 
\begin{align}
\begin{split}
d \rho 
& = Z^{-1} e^{- H(u, W, \vec V) - \frac 12 M(u)} du \, dW \, d \vec V\\
& = Z^{-1} e^{Q(u, W)} d\mu_{1}(u) d\mu_0(W) \prod_{j = 1}^d d \mu_0(V_j), 
\end{split}
\label{Zak6}
\end{align}

\noi
where
 the potential $Q(u, W)$ is given by 
\begin{align}
Q(u, W) = - \frac{1}{\sqrt 2} \int_{\T^d}
|u|^2 W dx, 
\label{Zak7}
\end{align}

\noi
the measure   $\mu_{1}$ denotes the complex-valued version
of the massive Gaussian free field on $\T^d$ with the density  formally given by  
\begin{align*}
d \mu_1 
= Z^{-1} e^{-\frac 12 \| u\|_{H^{1} }^2    } du
& =  Z^{-1} \prod_{n \in \Z^d} 
e^{-\frac 12 \jb{n}^2 |\ft u(n)|^2}   
d\ft u(n), 
\end{align*}

\noi
and $\mu_0$ denotes the white noise measure 
defined as  the pushforward measure
$\mu_0 = (\jb{\nb}^\frac{d}{2})_*\mu$, 
with  $\mu$  as in \eqref{gauss0}.
In view of the conservation of the Hamiltonian $H(u, W, \vec V)$
and the wave energy $M(u)$, 
the Gibbs measure $\rho$ in \eqref{Zak6}
expected to be invariant under the Zakharov dynamics.

As in the case of the focusing NLS, the main issue in constructing
the Gibbs measure $\rho$ in~\eqref{Zak6} comes
from the focusing nature of the potential, 
i.e.~the potential $Q(u, W)$ is unbounded from above.
In a seminal paper \cite{LRS}, 
Lebowitz, Rose, and Speer constructed the Gibbs measure $\rho$ when $d = 1$, 
by inserting a cutoff in terms of the conserved  wave energy $M(u) = \|u\|_{L^2}^2$,
which was then proved to be invariant under \eqref{Zak3} on $\T$ (and thus \eqref{Zak1})
by Bourgain \cite{BO94Zak}.

Then, a natural question is to consider the construction
of the Gibbs measure $\rho$ in the two-dimensional setting.\footnote{In a recent work \cite{Seong}, 
the second author studied  the construction of the Gibbs measure 
for the Zakharov-Yukawa system on $\T^2$ (i.e.~$\Dl$ in \eqref{Zak1} is replaced by $-(-\Dl)^{\g}$, $\g<1$)
and showed that the renormalized Gibbs measure is indeed normalizable
when $\g < 1$.
See \cite{Seong} for details.}
Before doing this, let us recall the relation between the Zakharov system
and the focusing cubic NLS.
By sending the wave speed $c$ in \eqref{Zak1} to $\infty$, 
the Zakharov system converges, at a formal level, 
to the focusing cubic NLS.
See, for example, \cite{OzT, MN} for rigorous convergence results on $\R^d$.
When $d = 2$, 
Theorem \ref{THM:1} states that the (renormalized) Gibbs measure
for the focusing cubic NLS on $\T^2$
is not normalizable, even with a Wick-ordered $L^2$-cutoff.
This suggests that, when $d = 2$,  the Gibbs measure
$\rho$ in \eqref{Zak6} for the Zakharov system
may not be constructible 
even with a Wick-ordered $L^2$-cutoff on the Schr\"odinger component $u$.

Given $N \in \N$, 
define 
 the following renormalized truncated potential energy:
\begin{align}
Q_N(u, W) = - \frac{1}{\sqrt 2} \int_{\T^2}
:\! | u_N|^2 \!: W dx
\label{Zak8}
\end{align}

\noi
where $u_N = \pi_N u $ as in Subsection \ref{SUBSEC:1.1}
and 
$:\! | u_N|^2 \!:  \, =  | u_N|^2  -\s_N$.
We then define  the renormalized truncated  Gibbs measure $\rho_{N}$
on $\T^2$, 
endowed with  a Wick-ordered $L^2$-cutoff,  by 
\begin{align*}
\begin{split}
d \rho_N 
& = Z_N^{-1} \ind_{ \{|\int_{\T^2} 
: | u_N|^2 :   dx| \le K\}}
e^{Q_N(u, W)} d\mu_{1}(u) d\mu_0(W) \prod_{j = 1}^2 d \mu_0(V_j), 
\end{split}
\end{align*}

\noi
By integrating in $(V_1, V_2)$ and then in $W$,
we have 
\begin{align}
\begin{split}
\iint & \ind_{ \{|\int_{\T^2} 
: | u_N|^2 :   dx| \le K\}}
e^{Q_N(u, W)}  d\mu_0(W) \prod_{j = 1}^2 d \mu_0(V_j) \\
& = \ind_{ \{|\int_{\T^2} 
: | u_N|^2 :   dx| \le K\}}
\int
\exp\bigg(-\frac{1}{\sqrt 2} \sum_{n \in \Z^2}
\F(: \!| u_N|^2 \!: )(n)\,  \cj{\ft W(n)} \bigg) d\mu_0(W)  \\
& = 
\ind_{ \{|\int_{\T^2} 
: | u_N|^2 :   dx| \le K\}}
\int_\R \exp\bigg(-\frac{1}{\sqrt 2} 
\F(: \!| u_N|^2 \!: )(0)\,   g_0\bigg)
\frac{e^{-\frac 12 g_0^2}}{\sqrt {2\pi}}
dg_0\\
& \hphantom{X}
\times \prod_{n \in \Ld} \frac{1}{\pi}\int_\C
\exp\bigg(-\sqrt 2 \Re\Big(
\F( | u_N|^2  )(n)\,  \cj{ g_n }\Big)
 \bigg) e^{-|g_n|^2}d g_n, 
\end{split}
\label{Zak10}
\end{align}

\noi
where 
$\{ g_n \}_{n \in \Z^2}$ is as in \eqref{IV2}\footnote{In particular, 
$g_0$ is a standard real-valued Gaussian random variables
where $\Re g_n$ and $\Im g_n$, $n \in \Ld$, 
are independent 
real-valued Gaussian random variables
with mean 0 and variance $\frac 12$.
}
and $\Ld$ denotes the index set given by 
$\Ld = (\Z\times\Z_+)\cup (\Z_+\times\{0\})$
such that $\Z^2 = \Ld \cup (-\Ld) \cup\{0\}$.
Here, we used the fact that 
$\F(: \!| u_N|^2 \!: )(n) = \F(| u_N|^2 )(n)$
for $n \ne 0$.
Then, recalling the moment generating function
$\E[e^{tX}] = e^{\frac 12 \s t^2}$
for $X \sim \NN_\R(0, \s)$,  we have
\begin{align}
\begin{split}
\eqref{Zak10} 
& = 
\ind_{ \{|\int_{\T^2} 
: | u_N|^2 :   dx| \le K\}}
 \exp\bigg(
\frac 14  \big(\F(: \!| u_N|^2 \!: )(0)\big)^2
 \bigg)\\
&  \hphantom{X}
\times \prod_{n \in \Ld} \frac{1}{\pi}\int_\C
\exp\bigg(-\sqrt 2
\Re \! \big(\F( | u_N|^2  )(n)\big) 
  \Re g_n
   \\
 & 
 \hphantom{XXXXXXXXXl}
  - \sqrt 2
\Im\! \big(\F( | u_N|^2  )(n) \big) \Im g_n
 \bigg) e^{-|g_n|^2}d g_n\\
 & \geq  
\ind_{ \{|\int_{\T^2} 
: | u_N|^2 :   dx| \le K\}}
 \exp\bigg( \frac 14 \big\| \pi_{\ne 0} |u_N|^2 \big\|_{L^2}^2 - CK^2 \bigg), 
 \end{split}
\label{Zak11}
\end{align}

\noi
where $\pi_{\ne 0}$ is the projection onto non-zero frequencies.

Let $\{ h_n \}_{n \in \Z^2}$
be a sequence of mutually independent standard complex-valued
Gaussian random variables.
Then, we have

\begin{align}
\begin{split}
\int\big\| \pi_{\ne 0} |u_N|^2 \big\|_{L^2}^2 d \mu_1
& = \E\Bigg[ \sum_{\substack{n_1 - n_2 + n_3 - n_4 = 0\\
|n_j| \leq N\\
n_1 - n_2 \ne 0}}
\frac{h_{n_1}}{\jb{n_1}}
\frac{\cj{h_{n_2}}}{\jb{n_2}}
\frac{h_{n_3}}{\jb{n_3}}
\frac{\cj{h_{n_4}}}{\jb{n_4}}\Bigg]\\
& = \sum_{|n_1 |\leq N}
\frac{1}{\jb{n_1}^2}
\sum_{\substack{|n_3|\leq N\\n_3 \ne n_1}}
\frac{1}{\jb{n_3}^2}
\sim (\log N)^2 \too \infty, 
\end{split}
\label{Zak11a}
\end{align}

\noi
as $N \to \infty$.
Then, from \eqref{Zak11a}, 
the interpolation of the $L^p$-spaces, and  Lemma \ref{LEM:hyp}, we have
\begin{align}
\begin{split}
\log N & \sim \big\| \| \pi_{\ne 0} |u_N|^2 \|_{L^2}\big\|_{L^2(\mu_1)} \ge \big \| \| \pi_{\ne 0} |u_N|^2 
\|_{L^2}\big\|_{L^1(\mu_1)}\\
&  \ge \frac{\big \| \| \pi_{\ne 0} |u_N|^2 \|_{L^2}\big\|_{L^2(\mu_1)}^3}{\big \| \| \pi_{\ne 0} |u_N|^2 \|_{L^2}\big\|_{L^4(\mu_1)}^2}  \sim \log N.
\end{split}
\label{Zak11b}
\end{align}

\noi
Also, from Lemma \ref{LEM:hyp}
and \eqref{sigma1},  
we have 
\begin{align}
\big\|\| u_N \|_{L^4_x}\big\|_{L^2(\mu_1)}
\les \big\|\| u_N \|_{L^2(\mu_1) }\big\|_{L^4_x}
\sim \s_N^\frac 12  \sim (\log N)^\frac 12.
\label{Zak11d}
\end{align}

\noi
Hence, 
given sufficiently small $ \eps \gg  \eta > 0$, 
it follows from 
Lemma \ref{LEM:var3}, 
Cauchy's inequality, Sobolev's inequality, 
\eqref{Zak11b}, 
and \eqref{Zak11d} that 
\begin{equation*}
\begin{split}
-\log & \bigg(\int   \exp\Big(-\eta\big\|\pi_{\ne 0} |u_N|^2 \big\|_{L^2}\Big) d \mu_1(u)\bigg)\\
& = \inf_{\dr \in \mathbb H_a} 
\E\bigg[  \eta \big\| \pi_{\ne 0} |\pi_N Y(1) + \pi_N I(\dr)(1)|^2\big\|_{L^2} + \frac{1}{2} \int_0^1 \| \dr(t) \|_{L^2_x}^2 dt \bigg]\\
& \ge \inf_{\dr \in \mathbb H_a} 
\E\bigg[  \eta \Big(\big\| \pi_{\ne 0} |\pi_N Y(1)|^2 \big\|_{L^2} - 2\| \pi_N Y(1) \pi_N I(\dr)(1)\|_{L^2} \\
& 
\hphantom{XXXXX}
- \|\pi_N I(\dr)(1)\|_{L^4}^2\Big)
+ \frac{1}{2} \int_0^1 \| \dr(t) \|_{L^2_x}^2 dt \bigg] \\
& \ge \inf_{\dr \in \mathbb H_a} \E\bigg[  \eta \Big(\big\| \pi_{\ne 0} |\pi_N Y(1)|^2 \big\|_{L^2} - \eps \| \pi_N Y(1)\|_{L^4}^2\Big) + \frac{1}{4} \int_0^1 \| \dr(t) \|_{L^2_x}^2 dt \bigg]\\
& \gtrsim \eta (\log N).
\end{split}
\end{equation*}

\noi
Therefore, 
 we obtain
\begin{align*}
\int  \exp\Big(-\eta\big\|\pi_{\ne 0} |u_N|^2 \big\|_{L^2}\Big) d \mu_1(u) 
\le \exp(-c\eta \log N)
\end{align*}

\noi
 for some constant $c > 0$.
Then, by Chebyshev's inequality, we conclude that, for any $M > 0$, 
\begin{equation} 
\begin{split}
\mu_1\Big( \big\|\pi_{\ne 0} |u_N|^2 \big\|_{L^2} > M \Big) \ge 1 - 
\exp\big(\eta(M- c \log N) \big)
\too 1, 
\end{split}
\label{Zak12}
\end{equation}

\noi
as $N \to \infty$.

 We also note that, given any $K > 0$,  there exists a constant $c_K > 0$ such that 
\begin{equation}\label{Zak13}
\begin{split}
\mu_1\bigg(\Big|\int_{\T^2} : | u_N|^2 :   dx \Big|\le   K\bigg) \ge  c_K,  
\end{split}
\end{equation}

\noi
uniformly in $N \in \N$.
Indeed, 
for $L = L(K) > 0$ (to be chosen later), as in \eqref{pax1}, we have 
\begin{align}
\E_{\mu_1}  \big[  e^{L}\cdot  \ind_{ \{|\int_{\T^2} 
: | u_N|^2 :   dx| \le K\}}\big]
\geq 
\E_{\mu_1}  \Big[  \exp\big(L\cdot  \ind_{ \{|\int_{\T^2} 
: | u_N|^2 :   dx| \le K\}}\big) \Big] - 1.
\label{Zak13a}
\end{align}

\noi
Now, by repeating the argument in 
Subsection~\ref{SUBSEC:3.2}, in particular,  \eqref{pa4} and \eqref{pa5} with $M = M_0(K)$, 
we have
\begin{align}
\begin{split}
 -\log & \, \E_{\mu_1}   \Big[  \exp\big(L\cdot  \ind_{ \{|\int_{\T^2} 
: | u_N|^2 :   dx| \le K\}}\big) \Big] \\
&\le \E\bigg[ -L\cdot 
\ind_{\{ |\int_{\T^d} ( :{Y_N^2}: + 2 Y_N \Dr^0 + (\Dr^0)^2) dx | \le K\}} + \frac 12 \int_0^1 \| \dr^0(t)\|_{L^2_x} ^2 dt \bigg] \\
&\le - \frac 12 L + C M_0^d \log M_0
\le - \frac 14 L 
\end{split}
\label{Zak13b}
\end{align}

\noi
by choosing $L = L(M_0) = L(K) \gg 1$ sufficiently large.
From \eqref{Zak13a} and \eqref{Zak13b}, 
we then obtain 
\begin{align*}
\mu_1\bigg(\Big|\int_{\T^2} : | u_N|^2 :   dx \Big|\le   K\bigg) 
\ge  \frac{e^{\frac 14 L} - 1}{e^L} =: c_K, 
\end{align*}

\noi
yielding \eqref{Zak13}.

Therefore, 
from 
\eqref{Zak10}, \eqref{Zak11}, \eqref{Zak12}, 
and \eqref{Zak13}, 
we obtain, for any $K > 0$, 
\begin{align*}
\begin{split}
 \lim_{N \to \infty} &  \iiint  \ind_{ \{|\int_{\T^2} 
: | u_N|^2 :   dx| \le K\}}
e^{Q_N(u, W)} d\mu_{1}(u) d\mu_0(W) \prod_{j = 1}^2 d \mu_0(V_j)\\
 & \ge 
 \liminf_{N \to \infty} 
\int
\ind_{ \{|\int_{\T^2} 
: | u_N|^2 :   dx| \le K\}}
 \exp\bigg( \frac 14 \big\| \pi_{\ne 0} |u_N|^2 \big\|_{L^2}^2 - CK^2 \bigg)
d\mu_{1}(u) 
\\
& \ge  \liminf_{N \to \infty}
 \Big(c_K - \exp\big(\eta(M- c \log N)\big) \Big)
 \exp\bigg(\frac14 M^2 - CK^2\bigg) \\
& = 
c_K 
 \exp\bigg(\frac14 M^2 - CK^2\bigg) 
 \too \infty 
\end{split}
\end{align*}
\noi

\noi
by taking $M \to \infty$.
This shows the non-normalizability of the Gibbs
measure for the Zakharov system on $\T^2$
even if we apply the Wick renormalization on the potential energy
$Q(u, W)$ in~\eqref{Zak7}
and endow the measure with a Wick-ordered $L^2$-cutoff
on the Schr\"odinger component.

Another way  would be to 
apply a change of variables 
as in the one-dimensional case due to  Bourgain \cite{BO94Zak}.
Namely, 
rewrite the Hamiltonian 
in \eqref{Zak5}
as in the one-dimensional case by Bourgain \cite{BO94Zak}:
\begin{align*}
H(u, W, \vec V) = \frac{1}{2} \int_{\T^2}
|\nb u|^2 dx -   \frac14 \int_{\T^2} |u|^4 dx 
+ \frac{1}{2} \int_{\T^2} (W+ \sqrt 2|u|^2)^2 dx 
+ \frac{1}{2} \int_{\T^2} |\vec V|^2 dx. 
\end{align*}

\noi
By introducing a new variable $\wt W = W+ \sqrt 2|u|^2$, 
we arrive at
\begin{align*}
\wt H(u, \wt W, \vec V) = \frac{1}{2} \int_{\T^2}
|\nb u|^2 dx -   \frac14 \int_{\T^2} |u|^4 dx 
+ \frac{1}{2} \int_{\T^2} \wt W^2 dx 
+ \frac{1}{2} \int_{\T^2} |\vec V|^2 dx. 
\end{align*}

\noi
Then, we apply the Wick renormalization to the potential energy. 

In this formulation, we 
consider  the renormalized truncated  Gibbs measure $\wt \rho_{N}$
defined by 
\begin{align*}
\begin{split}
d \wt \rho_N 
& = Z_N^{-1} \ind_{ \{|\int_{\T^2} 
: | u_N|^2 :   dx| \le K\}}
e^{R_N(u)} d\mu_{1}(u) d\mu_0(\wt W) \prod_{j = 1}^2 d \mu_0(V_j), 
\end{split}
\end{align*}

\noi
where the renormalized truncated potential energy
$R_N$ is defined by 
\begin{align*}
R_N(u)=\frac 14\int_{\T^2}  :\! |u_N|^4 \!:  dx.
\end{align*}

\noi
Note that, in the complex-valued setting,
the Wick-ordered fourth power is given by 
\[:\! |u_N|^4 \!: \, = |u_N|^4 - 4\s_N |u_N|^2 + 2\s_N^2.\]

\noi
See \cite{OTh}.
Then, by integrating in $\wt W$ and $\vec V$ and then 
by applying Theorem \ref{THM:1}
(in the complex-valued setting), we have 
\begin{align*}
\begin{split}
\sup_{N \in \N}
& \iiint
\ind_{\{|\int : | u_N|^2 :   dx| \le K\}}
e^{R_N(u)} d\mu_{1}(u) d\mu_0(\wt W) \prod_{j = 1}^2 d \mu_0(V_j)\\
& = 
\sup_{N \in \N}
\int
\ind_{\{|\int : | u_N|^2 :   dx| \le K\}}
e^{R_N(u)} d\mu_{1}(u) 
= \infty
\end{split}
\end{align*}

\noi
for any $K > 0$.
This shows the non-normalizability 
of the limiting Gibbs measure in this formulation.

\begin{remark}\rm
In the renormalization \eqref{Zak8}, 
we added the term $\frac {\s_N}{\sqrt2}\int_{\T^2}Wdx = \frac {\s_N}{2}\int_{\T^2}wdx$.
Note that the spatial mean of $w$ is conserved under the flow of the system \eqref{Zak3}.
Thus, by imposing the spatial mean-zero condition on $w$, 
we can write $Q_N(u, W)$ in  \eqref{Zak8} as
\begin{align*}
Q_N(u, W) = - \frac{1}{\sqrt 2} \int_{\T^2}
:\! | u_N|^2 \!: W dx
= - \frac{1}{\sqrt 2} \int_{\T^2}
 | u_N|^2  W dx, 
\end{align*}

\noi
showing that this term is self-renormalizing, 
and thus the renormalization \eqref{Zak8} does not affect the system
 \eqref{Zak3}.

\end{remark}

\section{Focusing quartic Gibbs measures with smoother Gaussian fields}
\label{SEC:B}

In this appendix,  we briefly discuss the construction 
of the focusing Gibbs measure $\rho_\al$ in~\eqref{Q2}
with a smoother base Gaussian measure $\mu_\al$
in \eqref{gauss1}.
We only discuss the  uniform exponential integrability bound
\eqref{Q1}.
Since the precise value of $\ld \in \R\setminus \{0\}$ does not play any role, 
we set $\ld = 4$ in the following.
As before, we also assume $p = 1$ for simplicity.

Fix  $\al > \frac d 2$.
The Gaussian measure  $\mu_\al$ in \eqref{Q2}
is the induced probability measure under the map:
\begin{equation*} 
\o\in \O \longmapsto u(\o) = \sum_{n \in \Z^d } \frac{ g_n(\o)}{\jb{n}^\al} e_n, 
\end{equation*}

\noi
where 
$\{ g_n \}_{n \in \Z^d}$ is as in \eqref{IV2}.
In particular,  a typical function $u$ in the support of $\mu$ 
belongs to $L^\infty(\T^d)$.

We define $Y^\al$ by 
\begin{align*}
Y^\al(t)
=  \jb{\nabla}^{-\al}W(t), 
\end{align*}

\noi
where  $W$ is  as in \eqref{P1}.
Then, 
in view of the Bou\'e-Dupuis formula (Lemma \ref{LEM:var3}), 
it suffices to  establish a  lower bound on 
\begin{equation}
\W_N^\al(\dr) = \E
\bigg[-R^{\diamond, \g}_N(Y^\al(1) + I^\al(\dr)(1)) + \frac{1}{2} \int_0^1 \| \dr(t) \|_{L^2_x}^2 dt \bigg], 
\label{Q7}
\end{equation}

\noi 
uniformly in $N \in \N$ and  $\dr \in \Ha$, 
where $R^{\diamond, \g}_N (u)$
and  $I^\al(\dr)$ are defined by 
\begin{align}
\begin{split}
R^{\diamond, \g}_N (u)
&=  \int_{\T^d}   u_N^4   dx
- A \, \bigg( \int_{\T^d}  u_N^2  dx\bigg)^\g
\end{split}
\label{Q8}
\end{align}

\noi
for some $\g >0$ (to be chosen later)
and 
\begin{align*}
 I^\al(\dr)(t) = \int_0^t \jb{\nabla}^{-\al} \dr(t') dt'.
\end{align*}

\noi
For simplicity of notation, 
we set 
 $Y_N^\al = \pi_N Y^\al = \pi_N Y^\al(1)$ and $\Dr_N^\al = \pi_N  \Dr^\al = \pi_N I^\al(\dr)(1)$.

From  \eqref{Q7} and \eqref{Q8}, we have
\begin{align*}
\begin{split}
\W_N^{\al}(\dr)
&=\E
\bigg[
-\int_{\T^d}   (Y_N^\al)^4  dx
-4\int_{\T^d}   (Y_N^\al)^3  \Dr_N^\al dx
-6\int_{\T^d}   (Y_N^\al)^2  (\Dr_N^\al)^2 dx\\
&\hphantom{XXX}
-4\int_{\T^d}   Y_N^\al  (\Dr_N^\al)^3 dx
-\int_{\T^d} (\Dr_N^\al)^4 dx
+  A \bigg\{ \int_{\T^d} \big (Y_N^\al + \Dr_N^\al\big)^2  dx \bigg\}^2\\
&\hphantom{XXX}
+ \frac{1}{2} \int_0^1 \| \dr(t) \|_{L^2_x}^2 dt 
\bigg].
\end{split}
\end{align*}

\noi
Let us first state a lemma, analogous to Lemma \ref{LEM:Dr2}.

\begin{lemma} \label{LEM:QQ}
\textup{(i)}
Let $\al > \frac d2$.
Then, there exists  $c  >0$ such that
\begin{align}
\bigg| \int_{\T^d}  ( Y_N^\al)^3 \Dr_N^\al dx \bigg|
&\le c \|  Y_N^\al \|_{L^\infty}^6
+ \frac 1{100} 
\| \Dr_N^\al \|_{L^2}^2, 
\label{QQ1} \\
\bigg| \int_{\T^d}  (Y_N^\al)^2 (\Dr_N^\al)^2  dx \bigg|
&\le c 
\| Y_N^\al \|_{L^\infty}^{4} + \frac 1{100} \| \Dr_N^\al \|_{L^2}^4,
\label{QQ2}\\
\bigg| \int_{\T^d} Y_N^\al (\Dr_N^\al)^3 dx \bigg|
&\le c \| Y_N^\al \|_{L^\infty}^{4}  
+ \frac 1{100} \| \Dr_N^\al \|_{L^4}^4, 
\label{QQ2a}\\
\bigg| \int_{\T^d}  (\Dr_N^\al)^4  dx \bigg|
&\le 
\frac 1{100} \| \Dr_N^\al \|_{H^\al}^2
+   \frac{A}{100} \| \Dr_N^\al \|_{L^2}^{\frac{8\al - 2d}{2\al - d}}
\label{QQ3}
\end{align}

\noi
for any  sufficiently large $A>0$,  
uniformly in $N \in \N$.

\smallskip

\noi
\textup{(ii)}	
Let $A, \g> 0$. 
Then, there exists $c = c(A, \g)>0$ such that
\begin{align}
\begin{split}
A\bigg\{ \int_{\T^d}&  
 \big (Y_N^\al + \Dr_N^\al\big)^2  dx \bigg\}^\g \ge \frac A4 \| \Dr_N^\al \|_{L^2}^{2\g} 
- c \| Y_N^\al \|_{L^\infty}^{2\g}, 
\end{split}
\label{QQ4}
\end{align}

\noi
uniformly in $N \in \N$.

\end{lemma}

Set
\begin{align}
\g = \frac {4\al- d}{2\al -d}.
\label{Q1a}
\end{align}

\noi
Then, 
by arguing as in Section \ref{SEC:4}
with Lemma \ref{LEM:QQ},\footnote{We bound the second term on the right-hand side
of \eqref{QQ2a} by \eqref{QQ3}.}
 the almost sure $L^\infty$-regularity of $Y^\al$,  and 
a variant of \eqref{CM} for 
$  \Dr^\al =  I^\al(\dr)(1)$:
\begin{align*}
\| \Dr^\al \|_{H^{\al}}^2 \leq \int_0^1 \| \dr(t) \|_{L^2}^2dt, 
\end{align*}

\noi
we obtain 
the following   uniform lower bound:
\begin{align}
\inf_{N \in \mathbb{N}} \inf_{\dr \in \Ha} \W_N^\al(\dr) 
\geq - C_0 >-\infty.
\label{Q12}
\end{align}

\noi
Then,  
 the uniform exponential integrability \eqref{Q1} 
follows from 
\eqref{Q12} and Lemma \ref{LEM:var3}.

We now present the proof of Lemma \ref{LEM:QQ}.

\begin{proof}[Proof of Lemma \ref{LEM:QQ}]
(i) The estimates \eqref{QQ1}, \eqref{QQ2}, and \eqref{QQ2a} follow from H\"older's and Young's inequalities.
As for the fourth estimate \eqref{QQ3}, 
it follows from 
Sobolev's inequality, Lemma \ref{LEM:prod}\,(i),  and Young's  inequality that 
\begin{align*}
\begin{split}
\bigg| \int_{\T^d}  (\Dr_N^\al)^4  dx \bigg|
&\le C  \| \Dr_N^\al \|_{H^\frac d4}^4 
\le C \| \Dr_N^\al \|_{H^\al}^\frac{d}{\al} \| \Dr_N^\al \|_{L^2}^{4 - \frac d\al}  \\
&\le
\frac 1{100}\| \Dr_N^\al \|_{H^\al}^2+ 
 \frac A{100} \| \Dr_N^\al \|_{L^2}^{\frac{8\al - 2d}{2\al - d}},
\end{split}
\end{align*}
	
\noi
for  sufficiently large $A > 0$.

\smallskip

\noi
(ii)  
Note that 
\begin{align}
|a+b+c|^\g
\ge \frac 12 |c|^\g - C_\g (|a|^\g+|b|^\g)
\label{QQ5}
\end{align}
	
\noi
for any $a,b,c \in \R$.
Then, 
the bound \eqref{QQ4} follows from \eqref{QQ5}
and 
\begin{align*}
\bigg| \int_{\T^d}   Y_N^\al \Dr_N^\al dx \bigg|^\g
&\le c \|  Y_N^\al \|_{L^\infty}^{2\g}
+ \frac 1{100C_\g} 
\| \Dr_N^\al \|_{L^2}^{2\g}.
\qedhere
\end{align*}
\end{proof}

\medskip

\begin{remark}\rm
Let $\g$ be as in \eqref{Q1a}.
Then, 
we have $\g > 2$.
Moreover, we have $\g \to \infty$ as $\al \to \frac d2+$, 
indicating an issue at $\al = \frac d2$
even if we disregard a renormalization  required
for $\al = \frac d2$.

\end{remark}

\begin{ackno}\rm
K.S.~would like to express his gratitude to the School of Mathematics at the University of Edinburgh for its 
hospitality during his visit, where this manuscript was prepared. 
The authors would like to thank the anonymous referees for the helpful comments which improved the quality of the paper.

\end{ackno}

\medskip

\noi
{\bf Conflict of interest.}
None.

\medskip

\noi
{\bf Financial Support.}
 T.O.~was supported by the European Research Council (grant no.~864138 ``SingStochDispDyn"). K.S.~was partially supported by National Research Foundation of Korea (grant NRF-2019R1A5A1028324). L.T.~was funded by the Deutsche
Forschungsgemeinschaft (DFG, German Research Foundation) under Germany's Excellence
Strategy-EXC-2047/1-390685813, through the Collaborative Research Centre (CRC) 1060.

\end{document}